\documentclass[twoside,a4paper,reqno,11pt]{amsart} 
\usepackage[top=28mm,right=30mm,bottom=28mm,left=30mm]{geometry}

\usepackage{amsfonts, amsmath, amssymb, mathrsfs, bm, latexsym, stmaryrd, array, mathtools, bold-extra}

\usepackage{etaremune}
\usepackage{enumitem}

\usepackage[]{hyperref}

\newcommand{\la}{\langle}
\newcommand{\ra}{\rangle}

\newcommand{\leqs}{\leqslant}

\newcommand{\Aut}{\operatorname{Aut}}

\newcommand{\PSL}{{\mathrm {PSL}}}
\newcommand{\GL}{{\mathrm {GL}}}
\newcommand{\PSp}{{\mathrm {PSp}}}
\newcommand{\PSU}{{\mathrm {PSU}}}
\newcommand{\PP}{{\mathrm {P}}}
\newcommand{\SL}{{\mathrm {SL}}}
\newcommand{\Sp}{{\mathrm {Sp}}}
\newcommand{\Spin}{{\mathrm {Spin}}}
\newcommand{\SU}{{\mathrm {SU}}}
\newcommand{\PGL}{{\mathrm {PGL}}}
\newcommand{\FF}{\mathbb{F}} 
\newcommand{\LL}{\mathrm {PSL}}
\newcommand{\UU}{\mathrm {PSU}}
\newcommand{\OO}{{\mathrm {O}}}

\newcommand{\PGU}{\mathrm {PGU}}

\newcommand{\Alt}{{\mathbf {A}}}
\newcommand{\Sym}{{\mathbf {S}}}

\newcommand{\SCC}{\mathrm{SCC}} 
\newcommand{\NCC}{\mathrm{NCC}} 
\newcommand{\CCC}{\mathrm{CCC}} 
 
\makeatletter
\newcommand{\imod}[1]{\allowbreak\mkern4mu({\operator@font mod}\,\,#1)}
\makeatother

\renewcommand{\leq}{\leqs}

\newtheorem{theorem}{Theorem} 
\newtheorem*{conj*}{Conjecture}

\newtheorem{corol}{Corollary}

\newtheorem{thm}{Theorem}[section] 
\newtheorem{prop}[thm]{Proposition} 
\newtheorem{lem}[thm]{Lemma}
\newtheorem{cor}[thm]{Corollary}

\theoremstyle{definition}
\newtheorem{rem}[thm]{Remark}
\newtheorem{remk}{Remark}

\newtheorem{problem}{Problem}

\begin{document}

\title[Variations on the Thompson theorem]{Variations on the Thompson theorem}

\author{Hung P. Tong-Viet}
\address{H.P. Tong-Viet, Department of Mathematics and Statistics, Binghamton University, Binghamton, NY 13902-6000, USA}
\email{htongvie@binghamton.edu}

\renewcommand{\shortauthors}{Tong-Viet}

\begin{abstract}
Thompson's theorem stated that a finite group $G$ is solvable if and only if every $2$-generated subgroup of $G$ is solvable. In this paper, we prove some new criteria for both solvability and nilpotency of a finite group using certain condition on $2$-generated subgroups. We show that a finite group $G$ is solvable if and only if for every pair of two elements $x$ and $y$ in $G$ of coprime prime power order, if $\langle x,y\ra$ is solvable, then $\la x,y^g\ra$ is solvable for all $g\in G$. Similarly, a finite group $G$ is nilpotent if and only if for every pair of elements $x$ and $y$ in $G$ of coprime prime power order, if $\la x,y\ra$ is solvable, then $x$ and $y^g$ commute  for some $g\in G.$ Some applications to graphs defined on groups are given.
\end{abstract}

\date{\today}

\maketitle

\setcounter{tocdepth}{1}
\tableofcontents


\section{Introduction}\label{s:intro}

As a consequence of the classification of finite $N$-groups, Thompson proved in \cite{Thompson}  that  a finite group $G$ is solvable if and only if every two-generated subgroup of $G$ is soluble.  A direct proof  of this theorem was given by P. Flavell in \cite{Flavell}. We obtain a generalization of this as follows.

\begin{theorem}\label{th:solvable}
Let $G$ be a finite group.  Then $G$ is solvable if and only if for every pair of distinct primes $p$ and $q$ and for every pair of elements $x,y\in G$ with $x$ a $p$-element and $y$ a $q$-element, if $\la x,y\ra$ is solvable, then $\la x,y^g\ra$  is solvable for all $g\in G$. 
\end{theorem}

It is easy to see that Theorem \ref{th:solvable} can be stated as follows: a finite group $G$ is solvable if and only if for every pair of distinct primes $p$ and $q$ and every pair of elements $x,y\in G$ with $x$ a $p$-element and $y$ a $q$-element, if $\la x,y^g\ra$ is solvable  for some $g\in G$, then $\la x,y\ra$ is solvable. The proof of Theorem \ref{th:solvable} depends only on Thompson's  classification of minimal simple groups in \cite{Thompson}. Numerous variations of Thompson's theorem mentioned earlier have been obtained in the literature. We mention a few here.

\smallskip
(i) Guralnick and Wilson \cite{GW}  obtain a probabilistic version of Thompson's theorem stating that if the probability that two randomly chosen elements of a finite group $G$ generate a solvable group is greater than $11/30$, then $G$ is solvable.

\smallskip
(ii) Guest  \cite{Guest10} and Gordeev et. al. \cite{GGKP09}, independently,  show that a finite group $G$ is solvable if and only if $\la x,x^y\ra$ is solvable for all $x,y\in G.$

\smallskip
(iii) Kaplan and Levy \cite{KL} prove that a finite group $G$ is solvable if and only if for all odd primes $p$, for all $p$-elements $x$ and $2$-elements $y$ in $G$, the subgroup $\la x,x^y\ra$ is solvable.

\smallskip
(iv) Dolfi, Guralnick, Herzog and Praeger \cite{DGHP} prove that a finite group $G$ is solvable if and only if for all $x,y\in G$ of prime power order, there exists an element $g\in G$  such that $\la x,y^g\ra$ is solvable.

\smallskip
(v)  Guralnick and Malle \cite{GM} (see also \cite{DGHP}) show that if $\mathfrak{X}$ is a family of finite groups closed under subgroups, quotients and extensions, then a finite group $G$ belongs to $\mathfrak{X}$ if and only if, for every $x,y\in G$, $\la x,y^g\ra\in \mathfrak{X}$ for some $g\in G.$

\smallskip
(vi) Guest and Levy \cite{GL13} show that a finite group $G$ is solvable if and only if for all odd primes $p$ dividing the order of $G$, except possibly one, all $p$-elements $x\in G$ and all $2$-elements $y\in G$, then  $\la x,x^y\ra$ is solvable.

Many of these results are actually consequences of much stronger statements about the characterization of solvable radical of finite groups. Recall that for a finite group $G$, the solvable radical of $G$, denoted by $R(G)$,  is the largest normal solvable subgroup of $G$.
We next prove a nilpotency criterion which is a generalization of Corollary E in \cite{DGHP}.

\begin{theorem}\label{th:nilpotent}
Let $G$ be a finite group.  Then $G$ is nilpotent if and only if for every pair of distinct primes $p$ and $q$ and for every pair of elements $x,y\in G$ with $x$ a $p$-element and $y$ a $q$-element, if $\la x,y\ra$ is solvable, then $x$ and $y^g$ commute for some $g\in G.$
\end{theorem}

The proof of Theorem \ref{th:nilpotent} depends on the full classification of finite groups. In particular, it depends on the existence of certain elements in finite simple groups.  An element $x$ of a finite group $G$ is  real  if $x^g=x^{-1}$ for some $g\in G$. The element $g\in G$ inverting a real element $x$ can be chosen to be a $2$-element. The existence or non-existence of real elements with certain property is an important question in finite group theory. In \cite{DNT}, it is shown that a finite group $G$ has a normal Sylow $2$-subgroup if and only if $G$ has no nontrivial real elements of odd order.  Suzuki \cite{Suzuki} determines the structure of finite groups $G$ whose all involution centralizers have normal Sylow $2$-subgroups. As it turns out  this condition is equivalent to the fact that $G$ has no real element of order $2m$, where $m>1$ is an odd integer. The groups with these properties were further investigated in \cite{DGN}.

It follows from \cite{DNT} that every finite nonabelian simple group has a nontrivial real element of odd order.  In the next result, we show that with some exceptions, every nonabelian simple group contains a real element of odd order and odd centralizer.

\begin{theorem}\label{th:odd-centralizer}
Let $S$ be a nonabelian simple group. Then 

\smallskip
{\rm{(i)}} $S$ has a nontrivial real element  of odd  prime power order whose centralizer has odd order; or

\smallskip
{\rm(ii)} $S$ is isomorphic to  ${\rm M}_{12}, {\rm Co}_{1}$  or the alternating groups $\Alt_n$ with $n\ge 42$ or $n=24.$ In these cases, there exist two elements $x,y\in S$, where $x$ is a $p$-element for some odd prime $p$ and $y$ is a $2$-element, such that $y$ normalizes $\la x\ra$ but no conjugate of $y$ in $\Aut(S)$ commutes with $x$.
\end{theorem}

The proof of this theorem depends on the classification of finite simple groups, in particular, results in \cite{GNT, MT,TZ04,TZ05} concerning the existence of real elements and their centralizers in finite simple groups of Lie type.
\begin{remk}\label{rem1} We record here the following observations.

\smallskip
(i) The  sporadic simple Conway group $\textrm{Co}_1$  has a self-centralizing real element of order $35$  but it does not have any real element  of prime power order with an odd centralizer. 

\smallskip
(ii)  Every real element of $\textrm{M}_{12}$ has an even order centralizer.

\smallskip
(iii) For all integers $5\leq n\leq 41$, only the  alternating group $\Alt_{24} $ of degree $24$ does not have a prime power real element with an odd centralizer.   

\smallskip
(iv) The alternating group $\Alt_n$ with $n\ge 5$, always has a real element of odd order with an odd centralizer.
\end{remk}

Recall that a finite group $G$ is an almost simple group with socle $S$ if there exists a finite nonabelian simple group $S$ such that $S\unlhd G\leq \Aut(S)$. 
We deduce the following consequence of Theorem \ref{th:odd-centralizer}.

\begin{theorem}\label{th:almost simple}
Let $G$ be a finite almost simple group with socle $S$. Then $S$ has a nontrivial $p$-element $x$, for some odd prime $p$, and a nontrivial $2$-element $y$, such that $y$ normalizes $\la x\ra$ and hence $\langle x,y\ra$ is solvable but $x$ does not commute with any $G$-conjugate of $y$.
\end{theorem}

We now use Theorem \ref{th:almost simple} to establish the following solvability criterion.

\begin{theorem}\label{th:direct-factor}
Let $G$ be a finite group. Assume that for every pair of two nontrivial elements $x$ and $y$, where $x$ is a $p$-element for some odd prime $p$, and $y$ is a $2$-element, if $\la x,y\ra$ is solvable, then $x$ commutes with some $G$-conjugate of $y$. Then a Sylow $2$-subgroup of $G$ is a direct factor of $G$.
\end{theorem}

Before we can present some applications of Theorems \ref{th:solvable} and \ref{th:nilpotent}, we need to define certain graphs attached to a finite group $G$.
For  $x\in G$, we write  $x^G=\{g^{-1}xg:g\in G\} $ for the conjugacy class of $G$ containing $x$. In their seminal paper \cite{BF}, Brauer and Fowler define the commuting graph of a finite group $G$ as a finite simple graph whose vertex set is the set of all nontrivial elements of $G$ and two distinct vertices $a$ and $b$ are joined if $a$ and $b$ commute. Similarly, one can define the solvable graph of $G$ as a simple graph whose vertex set is the set of all nontrivial elements of $G$ and there is an edge between two distinct vertices $a$ and $b$ if $\la a,b\ra$ is solvable. The nilpotent graph can be defined in the same way. These graphs have been studied extensively in the literature.  For recent results on the solvable graphs, see \cite{ALMM, BLN}.

Generalizing the commuting graph using  conjugacy classes, Herzog, Longobardi and Maj   \cite{HLM} define  the commuting conjugacy class graph ($\textrm{CCC}$-graph) of $G$ as a simple graph whose vertex set is the set of nontrivial conjugacy classes of $G$ and two distinct vertices $C$ and $D$ are adjacent if $\langle c,d\rangle$ is abelian  for some $c\in C$ and $d\in D.$ 
The solvable conjugacy class graph ($\SCC$-graph) and the nilpotent conjugacy class graph ($\NCC$-graph) can be defined similarly (see \cite{BCNS}). Finally, following \cite{BCNS}, one can define a variant of these graphs as follows. The expanded $\SCC$-graph is a simple undirected graph whose vertex set is the set of nontrivial elements of $G$ and two distinct elements $x,y\in G$ are adjacent if $x^G=y^G$ or $x^G$ and $y^G$ are adjacent in the $\SCC$-graph of $G$.    The expanded $\NCC$-graph and the expanded  $\CCC$-graph on a group can be defined similarly. Note that these graphs are called conjugacy super graphs in \cite{ACN}. Interested readers are referred to recent survey papers \cite{ACN,Cameron}  for more information about these graphs.

As direct consequences of Theorems \ref{th:solvable} and \ref{th:nilpotent}, we obtain the following corollaries. The first two corollaries give a positive answer to Problem 2.4 in \cite{BCNS}.

\begin{corol}\label{cor:SCC-graph}
Let $G$ be a finite group. The expanded $\SCC$-graph of $G$ is equal to the solvable graph of $G$ if and only if $G$ is solvable.
\end{corol}

\begin{corol}\label{cor:NCC-graph}
Let $G$ be a finite group. The expanded $\SCC$-graph of $G$ is equal to the expanded $\NCC$-graph of $G$ if and only if $G$ is nilpotent.
\end{corol}

A finite group $G$ is said to be invariably generated by two elements $x,y\in G$ if $G=\la x^a,y^b\ra$ for all $a,b\in G.$ Invariable generation attracts much attention recently as it has applications to computational Galois theory. It is known that every nonabelian simple group can be invariably generated by two elements (see,  \cite{GM, KLS}). One can define an invariable generating graph on a finite group $G$ as follows. The vertex set of this graph is the set of nontrivial elements of  $G$ and two vertices are adjacent if and only if they invariably generate $G$. In \cite[Problem 1]{ACN}, the authors show that if $G$ is a non-nilpotent (non-abelian) group, then the invariable generating graph is contained in the complement of the expanded $\NCC$-graph (or $\CCC$-graph) and they ask when the equality holds. The same question was asked for the expanded $\SCC$-graph. The next two corollaries partially address these questions.

\begin{corol}\label{cor:invariable-gen-graph}
Let $G$ be a finite non-nilpotent group.  Then $G$ is minimal non-nilpotent if and only if  for all $x,y\in G$ of coprime prime power order, if $x$ and $y$ do not generate $G$, then $x$ and $y^g$ commute for some $g\in G$. 
\end{corol}

This extends \cite[Proposition 11.1 (e)]{Cameron}, where it is shown that for finite non-nilpotent groups $G$, the generating graph is equal to the complement of the nilpotent graph if and only if $G$ is a minimal non-nilpotent group. Recall that a finite group $G$ is called minimal non-nilpotent or a Schmidt group if $G$ is non-nilpotent but all of its proper subgroups are nilpotent.

\begin{corol}\label{cor:invariable-gen-SCC}
Let $G$ be a non-solvable finite group.  Suppose that for all $x,y\in G$  of coprime prime power order, if $x$ and $y$ do not generate $G$, then $\la x,y^g\ra$ is solvable for some $g\in G$.  Then $G/R(G)$ is a finite nonabelian simple group.
\end{corol}


The paper is organized as follows. After the Introduction, we prove  Theorems \ref{th:odd-centralizer} and \ref{th:almost simple} in Section \ref{sec2}.   Theorems \ref{th:solvable}, \ref{th:nilpotent} and \ref{th:direct-factor}  will be proved in Section \ref{sec3}. The corollaries will be proved in Section \ref{sec4}. In the last section, we prove a characterization of the solvable radical.

\section{Real elements in simple groups}\label{sec2}

Let $\Omega$ be a nonempty set of size $n\ge 1$. We denote by $\Sym_n=\textrm{Sym}(\Omega)$ and $\Alt_n=\textrm{Alt}(\Omega)$ the symmetric and alternating groups of degree $n$ acting on the set $\Omega$. Denote by $\mathcal{L}$ be the set consisting of the following simple groups: $\textrm{M}_{12},\textrm{Co}_1$ and $\Alt_n$ with $n\ge 42$ or $n=24.$

In this section, we will prove Theorems \ref{th:odd-centralizer} and \ref{th:almost simple}. We first start with the proof of Theorem \ref{th:odd-centralizer}, which we restate here. 

\begin{thm}\label{th:simple} Let $S$ be a finite nonabelian simple group. Then
\begin{itemize}
\item[$(1)$] If $S\not\in \mathcal{L}$, then $S$ has a real element $x$ of odd prime power order such that $|C_S(x)|$ is odd.
\item[$(2)$] If $S\in\mathcal{L}$, then $S$ has two nontrivial elements $x,y$, where $x$ is an element of odd prime order  and $y$ is a $2$-element such that $y$ normalizes the cyclic group $\la x\ra$ but no conjugate of $y$ in $\Aut(S)$ commutes with $x$. 
 \end{itemize}
 \end{thm}
 
 We will prove Theorem \ref{th:simple} through a series of lemmas.
 
\begin{lem}\label{lem:non-lie}
Theorem \ref{th:simple} holds for all sporadic simple groups, the Tits group ${}^2{\rm F}_4(2)'$ and the alternating groups of degree at least $5$.
\end{lem}

\begin{proof}
Assume that $S$ is a sporadic simple group and $S\not\in \mathcal{L}$. Then the statement can be checked easily  for sporadic simple groups and  the Tits group by using GAP \cite{GAP}. In Table \ref{tab:spor}, we list the triples $(S,x,|C_S(x)|)$, where $x$ is a real element of  odd prime power order and $|C_S(x)|$ is odd. 

Assume $S={\rm M}_{12}$. Then $S$ has only one conjugacy class of elements of order $5$ and so if $x$ is any such element, then $x$ is real. By using \cite{GAP}, we see that $C_G(x)\cong \textrm{D}_{10}$ and we can find an element  $y\in N_S(\la x\ra)\cong $ of order $4$ such that $x^y=x^2$. Since $C_S(x)$ has no element of order $4$, no $\Aut(S)$-conjugate of $y$ could lie in $C_S(x)$.

Assume that $S={\rm Co}_1$. By \cite[Table 5.3l]{GLS}, $S$ has a real element $x$ of order $13$ and a $2$-element $y$ of order $4$ inverting $x$ and $C_S(x)=\la x\ra\times \Alt_4$. Note that $\Aut(S)=S$ and that $C_S(x)$ has no element of order $4$, so no $\Aut(S)$-conjugate of $y$ lies in $C_S(x)$.

Finally, if $5\leq n\leq 41$ and $n\neq 24$, then by using GAP~\cite{GAP}, we can see that $S$ always contains a real element of odd prime power order such that $|C_S(x)|$ is odd. Assume $n=24$. Let $p=19$ and let $x\in S$ be a cycle of length $p.$  Then $S$ has an involution $y$ which is a product of $10$ disjoint transpositions which normalizes the cyclic group $\la x\ra$ but no conjugate of $y$ in $\Sym_n$ can commute with $x$. 

Now assume that $n\ge 42.$  It is well-known that there exist two primes $p$ and $q$ such that $n/2\leq p<q\leq n$ (see, for example, the proof of \cite[Proposition 3.3]{DGHP}). Clearly, both $p,q$ are odd and thus $n-p\ge 2.$ Let $\Omega=\{1,2,\cdots,n\}$, $\Omega_1=\{1,2,\cdots,p\}$ and $\Omega_2=\Omega-\Omega_1.$ Let $x=(1,2,\dots,p)\in\textrm{Sym}(\Omega_1)$ be a cycle of length $p.$ Let $\alpha\in \{1,2,\dots,p-1\}$ be a primitive root modulo $p$ and let $t\in \textrm{Sym}(\Omega_1)$ be a $(p-1)$-cycle such that $x^t=x^\alpha.$ Then $\la x,t\ra$ is the normalizer of the cyclic group $\la x\ra$ in $\textrm{Sym}(\Omega_1).$ Since $|\Omega_2|\ge 2$, we see that $s:=t(p+1,p+2)\in \Alt_n$ and that $C_S(x)=\langle x\ra\times \textrm{Alt}(\Omega_2).$ Write $p-1=2^am$, where $a,m\ge 1$ are integers and $m$ is odd. Let $y=s^m.$ Then $y\in S$ normalizes $\la x\ra$ but does not centralize it. Moreover, $y$ is a product of $m$ $2^a$-cycles and a transposition $(p+1,p+2)$. In particular, $y$ moves $p+1$ points in $\Omega$. Suppose by contradiction that $y^g$ commutes with $x$ for some $g\in \Sym_n=\Aut(\Alt_n)$ (note that $n\ge 42$), then $y^g\in C_S(x)$ and has the same cycle structure as that of $y$. Since $y^g$ is a $2$-elements, $y^g$ must lie in $\textrm{Alt}(\Omega_2)$, which is impossible since $|\Omega_2|=n-p<p+1.$
This completes the proof of the lemma.
\end{proof}

\begin{proof}[\textbf{Proof of the claims in Remark \ref{rem1}}.] The claim for sporadic simple groups ${\rm M}_{12}$, ${\rm Co}_1$ and the alternating group $\Alt_{24}$ can be checked easily using \cite{GAP}.  We next consider the alternating groups of degree
 $n\ge 5$. First assume that $n$ is odd. Then $n-2\ge 3$ is odd and so $n-3=2k$ is even. Let $x=(1,2,\dots,n-2)\in \Alt_n$ be a cycle of length $n-2$. Let $t=(1,2k+1)(2,2k)\dots (k,k+2)$ if $k$ is even; and $t=(1,2k+1)(2,2k)\dots (k,k+2)(2k+2,2k+3)$ if $k$ is odd. Thus $t$ is a product of an even number of transpositions, so $t\in \Alt_n$ and we can check that $x^t=x^{-1}$. So $x$ is a real element of odd order  $n-2$ and $C_S(x)=\la x\rangle$ is of odd order as wanted.
Next assume that $n\ge 6$ is even. Then $n-3=2k+1$ is odd and thus the elements $x$ and $t$ as constructed above are both in $\Alt_n$ and $C_S(x)=\la x\ra\times \la (n-2,n-1,n)\ra\cong C_{n-3}\times C_3$ is of odd order.  
\end{proof}

\begin{rem}
Although some alternating groups $\Alt_n$ with $n\ge 5$ might not contain a real element of prime power order with odd centralizer, it could be true that $A_n$ contains a real element $x$ of odd prime power order and a $2$-element $y$ such that $x^y=x^k\neq x$ for some positive integer $k$ and $x$ commutes with no $\Aut(\Alt_n)$-conjugate of $y.$
\end{rem}

{\scriptsize
\begin{table}[h]
\renewcommand\thetable{A}
\[
\begin{array}{llr|llr} \hline
S & x & |C_S(x)| & S & x& |C_S(x)| \\ \hline
{\rm M}_{11} & 5A & 5  & {\rm Th} & 19A & 19 \\
{\rm M}_{22} & 5A & 5  & {\rm Fi}_{22} & 13A & 13 \\
{\rm M}_{23} & 5A & 15  & {\rm Fi}_{23}  & 17A & 17 \\ 
{\rm M}_{24} & 11A & 11  & {\rm Fi}_{24}' & 29A & 29 \\
{\rm HS} & 7A & 7 &  {\rm O'N} & 19A & 19 \\
{\rm J}_{2} & 7A & 7  & {\rm J}_3 & 17A & 17 \\
{\rm Co}_{1} & 35A & 35  & {\rm Ru } & 29A & 29 \\
{\rm Co}_2 & 11A & 11 &  {\rm J}_4& 43A &43  \\
{\rm Co}_3 & 9B & 81 &  {\rm Ly}&67A & 67  \\
{\rm McL} & 5B & 25 &  {\rm J}_1 & 19A & 19 \\
{\rm Suz} & 13A & 13 & {\rm B} & 27A &27  \\
{\rm He} & 17A & 17 &  {\rm M}& 41A & 41 \\
{\rm HN} & 25A & 25  &{}^2{\rm F}_4(2)' &13A &13 \\\hline
\end{array}
\]
\caption{The triples $(S,x,|C_S(x)|)$ in Lemma \ref{lem:non-lie} with $S$ sporadic or the Tits group and $S\not\cong {\rm M}_{12}$}
\label{tab:spor}
\end{table}}

{\scriptsize
\begin{table}[h]
\renewcommand\thetable{B}
\[
\begin{array}{lcr|lcr|lcr} \hline
n & x & |C_{\Alt_n}(x)| &n &x &  |C_{\Alt_n}(x)|&n &x &  |C_{\Alt_n}(x)| \\ \hline

5 & 5A& 5&6&5A&5&7&5A&5\\

8 & 5A& 15&9&7A&7&10&7A&21\\

11 & 9A& 9&12&9A&27&13&13A&13\\

14 & 13A& 13&15&13A&13&16&13A&39\\

17 & 17A& 17&18&17A&17&19&17A&17\\

20 & 17A& 51&21&19A&19&22&19A&57\\

23 & 11B& 121&25&23A&23&26&25A&25\\

27 & 25A& 25&28&25A&75&29&29A&29\\

30 & 29A& 29&31&29A&29&32&29A&87\\

33 & 31A& 31&34&31A&93&35&25C&625\\

36 & 25C& 625&37&37A&37&38&37A&37\\

39 & 37A& 37&40&37A&111&41&41A&41\\
\hline
\end{array}
\]
\caption{The triples $(S,x,|C_S(x)|)$ in Lemma \ref{lem:non-lie} with $S$ an alternating $\Alt_n$ and $5\leq n\leq 41, n\neq 24$}
\label{tab:alt}
\end{table}}

We next consider the finite simple groups of Lie type. Recall that for a positive integer $n,$ we denote by $\Phi_n(x)$, the $n$th cyclotomic polynomial. For a positive integer $n\ge 2$ and a prime power $q$, the primitive prime divisor (or ppd ) of $q^n-1$ is a prime $\ell$ such that $\ell$ divides $q^n-1$ but $\ell$ does not divide $q^k-1$ for any integers $1\leq k<n.$ The Zsigmondy theorem \cite{Zsig} states that the primitive prime divisor of $q^n-1$ always exists except for when $n=2$ and $q=2^{f}-1$ or $n=6$ and $q=2$.
We will write $\ell_n=\ell_n(q)$ for a primitive prime divisor of $q^n-1$, if it exists. Since $n$ is the order of $q$ modulo $\ell_n(q)$, we have $\ell_n(q)\ge n+1.$

We will need the following observations.

\begin{lem}\label{lem:lem1}
Let $L$ be a finite group with center $Z=Z(L)$. Let $S=L/Z$.  Let $g\in L$ be a real element of odd order of $L$. Then

\begin{itemize}
\item[\rm (1)] $gZ$ is a real element of odd order of $S$.

\item[\rm (2)] If $|C_L(g):Z|$ is odd and either $\gcd(|g|,|Z|)=1$ or $|Z| $ is odd, then $|C_{S}(gZ)|$ is odd.
\end{itemize}
\end{lem}

\begin{proof}
(1) is obvious. For (2), as in the proof of Lemma 4.3 in \cite{GNT}, let $D$ be the full inverse image of $C_S(gZ)$ in $L$ and let $f:D\rightarrow Z$ be a map defined by $f(d)=g^{-1}dgd^{-1}$ for $d\in D$. Then $f$ is a group homomorphism with kernel $C_L(g)$. Hence $$|C_S(gZ)|=|D:Z|=|\textrm{Im}(f)|\cdot |C_L(g):Z|.$$ Thus $|C_S(gZ)|$ is odd if $|Z|$ is odd. On the other hand, if $\gcd(|g|,|Z|)=1$, then $|\textrm{Im}(f)|=1$ and so $|C_S(gZ)|=|C_L(x):Z|$ is odd. The proof is now complete.
\end{proof}

Recall that a finite group $L$ is quasisimple if $L$ is perfect, that is, $G'=[L,L]=L$, where $L'$ denotes the derived subgroup of $L$, and $L/Z(L)$ is a nonabelian simple group. 

\begin{rem} We record here some remarks which might be helpful for the remaining proofs.

\begin{itemize}
\item[\rm (1)] Let $L$ be a quasisimple group and let $S=L/Z(L)$. If $L$ has a real element $g$ such that $|C_L(g)|$ is odd, then $gZ$ is a real element of $S$ and $|C_S(gZ)|$ is odd by the previous lemma. In this case, $|Z(L)|$ is odd.
\item[\rm (2)] In the case when $|C_L(g)|$ is even for any real element $g\in G$, we find a real element $g\in L$ such that $(|g|,|Z|)=1$ and $|C_L(g):Z(L)|$ is odd, then $|C_S(gZ)|$ is odd and $gZ$ is real in $S$.
\end{itemize}
\end{rem}

We now prove Theorem  \ref{th:simple} for finite simple groups of Lie type.
\begin{prop}\label{prop:Lie type}
Let $S$ be a finite simple group of Lie type. Then $S$ has a nontrivial real element of  prime power order whose centralizer has odd order.
\end{prop}

We start with the exceptional groups of Lie type first. The proofs depend on \cite{MT,TZ05} and also \cite{GNT}. We write $\textrm{E}^+_6(q)$ for $\textrm{E}_6(q)$ and $\textrm{E}^-_6(q)$ for ${}^2\textrm{E}_6(q)$.  

\begin{lem}\label{lem:exceptional}
Proposition \ref{prop:Lie type} holds for all finite simple exceptional groups of Lie type.
\end{lem}

\begin{proof} Let $G$ be a finite quasisimple exceptional group of Lie type defined over a finite field  of size $q=p^f,$ where $p$ is a prime and $f\ge 1$ is an integer. Let $S=G/Z(G)$ and $d=|Z(G)|.$ Except for the groups ${\rm E}_6(q),{\rm E}_7(q)$ and ${}^2{\rm E}_6(q)$ whose center has order $d=(3,q-1),(2,q-1)$ and $(3,q+1),$ respectively, the centers of the other groups are trivial and so $G=S$ in these cases. We consider each family separately.

(i)  $G={}^2{\rm B}_2(q)$ with $q=2^{2f+1},f\ge 1.$ By Lemma 2.3 in \cite{MT}, $G$ contains a semisimple  element $s_1$ of order $p_1$ which is a ppd of $p^{4(2f+1)}-1$ and $|C_G(s_1)|=2^{2f+1}\pm 2^{f+1}+1$ is odd. Now, by Proposition 3.1 in  \cite{TZ05}, $s_1$ is real. Thus we can take $x=s_1.$

(ii)  $G={}^2{\rm G}_2(q)$ with $q=3^{2f+1},f\ge 1.$ By Lemma 2.3 in \cite{MT}, $G$ contains a semisimple  element $s_1$ of order $p_1$ which is a ppd of $3^{6(2f+1)}-1$ and $|C_G(s_1)|=3^{2f+1}\pm 3^{f+1}+1$ is odd. Now, by Proposition 3.1 in  \cite{TZ05}, $s_1$ is real. 

(iii)  $G={}^2{\rm F}_4(q)$ with $q=2^{2f+1},f\ge 1.$ By Lemma 2.3 in \cite{MT}, $G$ contains a semisimple  element $s_2$ of order $p_2$, a ppd of $2^{6(2f+1)}-1$ and $|C_G(s_2)|=q^2-q+1$ is odd. Now, by Proposition 3.1 in  \cite{TZ05}, $s_2$ is real.

(iv)  $G={}^3{\rm D}_4(q)$ with $q=p^{f},f\ge 1.$ Assume first that $q>2.$ By Lemma 2.3 in \cite{MT}, $G$ contains a semisimple element  $s_1$ of order $p_1$, a ppd of $p^{6f}-1$ and $|C_G(s_1)|=q^4-q^2+1$ is odd. Now, by Proposition 3.1 in  \cite{TZ05}, $s_1$ is real. For $q=2$, ${}^3{\rm D}_4(2)$ has a real element $x$ in the class $13A$ which is self-centralizing, that is,  $C_G(x)=\la x\ra.$

(v)  $G={\rm G}_2(q)$ with $q=p^{f},f\ge 1.$ Assume first that $q\ge 5.$ By Lemma 2.3 in \cite{MT}, $G$ contains a semisimple  element $s_1$ of order $p_1$, a ppd of $p^{6f}-1$ and $|C_G(s_1)|=\Phi_6(q)$. Now, $s_1$ is real by \cite{TZ05}. Since $\Phi_6(q)=q^2-q+1$ is odd, we can take $x=s_1.$

Note that ${\rm G}_2(2)\cong \PSU_3(3)\cdot 2$ is not simple but ${\rm G}_2(2)'\cong \PSU_3(3)$ is simple. By using GAP \cite{GAP}, ${\rm G}_2(2)'$ has a real element $x$ in the class $3B$  with $|C_G(x)|=9.$ For $q\in \{3,4\}$, ${\rm G}_2(q)$ has a real element of order $13$ which is self-centralizing.

(vi)  $G={\rm F}_4(q)$ or ${\rm E}_8(q)$ with $q=p^{f},f\ge 1.$  By Lemma 2.3 in \cite{MT}, $G$ contains a semisimple  element $s_1$ of order $p_1$, a ppd of $p^{m_1f}-1$,  with $|C_G(s_1)|=\Phi_{m_1}(q)$ with $m_1=12$ or $30$, respectively. By  \cite{TZ05} $s_1$ is real and clearly both $\Phi_{12}(q)=q^4-q^2+1$ and $\Phi_{30}(q)=q^8+q^7-q^5-q^4-q^3+q+1$ are odd.

(vii)  $G={\rm E}_7(q)$  with $q=p^{f},f\ge 1.$ Let $d=\gcd(2,q-1)$. Then $G$ has two cyclic maximal tori $T_1$ and $T_2$ of  order $(q^7-1)$ and $(q^7+1),$ respectively. Let $\ell_1$ and $\ell_2$ be ppd of $q^7-1$ and $q^7+1$ and let $s_i\in T_i$ be elements of order $\ell_i.$ As in the proof of Theorem 3.4 in \cite{BPS}, we have $C_G(s_i)=T_i$ and $|N_G(T_i)/T_i|=14$. It follows that both $s_i$ are real elements in $G$.  This also follows from Proposition 3.1 in \cite{TZ05} as $s_i's$ are semisimple in $G$. Since $|T_1|/d$ and $|T_2|/d$ are coprime, one of them must be odd. Moreover, both $\ell_i's$ are odd, so they are coprime to $d.$ The result now follows from Lemma \ref{lem:lem1}.

(viii)  $G={\rm E}^\epsilon_6(q)$ with $q=p^{f},f\ge 1$ and $\epsilon=\pm.$ Let $d=\gcd(3,q-\epsilon1)$. As in the proof of Theorem 4.1 in \cite{GNT}, by embedding $H={\rm F}_4(q)$ in $G$, $G$ has a semisimple real element $s_\epsilon\in  H$ of order $\ell$ with $|C_H(s_\epsilon)|=\Phi_{12}(q)$ and $|C_G(s_\epsilon)|=\Phi_{12}(q)(q^2+\epsilon q+1),$ where $\ell$ is a ppd of $q^{12}-1.$
Since $\Phi_{12}(q)=q^4-q^2+1$, $|C_G(s_\epsilon)|$ is odd. Since $d$ is also odd, the result follows from Lemma \ref{lem:lem1}. The proof is complete.
\end{proof}

If $m$  is a positive integer and $r$ is a prime,  we write $(m)_r$ for the $r$-part of $m$, that is, the largest power of $r$ dividing $m$.

\begin{lem}\label{lem:L}
Proposition \ref{prop:Lie type}  holds for $\LL_n(q)$, where $q=p^f$ and $n\ge 2.$
\end{lem}

\begin{proof} Let $G=\SL_n(q)$ and $d=|Z(G)|=\gcd(n,q-1)$. Let $G^*=\PGL_n(q).$ Then $S=G/Z(G)\cong [G^*,G^*]\cong  \LL_n(q).$

(1) Assume  $n=2$ and $q=p^f\ge 4.$ Since $\PSL_2(4)\cong \PSL_2(5)\cong\Alt_5$ and $\PSL_2(9)\cong\Alt_6$, we assume $9\neq q\ge 7$. 

Assume first that $p=2$. Then $q\ge 8$. In this case, $G=S$ has a subgroup isomorphic to the dihedral groups $D_{2(q-1)}$ and so $S$ has a real element $x$ of odd order $2^f-1 $ and  $|C_S(x)|=2^f-1$ is odd.

Assume that $q$ is odd. Using GAP \cite{GAP}, we may assume that $q\ge 13$. It is well known that $S$ has two  maximal subgroups isomorphic to $D_{q\pm 1}$ and the Borel subgroup $B=N_S(P)$ of order $q(q-1)/2$, where $P$ is a Sylow $p$-subgroup of $G$. Note that $B$ is a Frobenius group with Frobenius kernel $P$. If $q\equiv 1$ mod $4$, then $(q-1)/2$ is even and thus $B$ has a real element $x$ of order $p$ and $C_S(x)=P$ is of odd order and we are done. Next, assume that $q\equiv 3$ mod $4$. Then $(q-1)/2$ is odd. Hence the  dihedral group $D_{q-1}$  contains a real element $x$ of order $(q-1)/2$ and $|C_S(x)|=(q-1)/2$ is odd.

(2) Assume that $n=3$. Since $d=\gcd(3,q-1)$ is odd, it suffices to find a real element $g\in G$ with $|C_G(g)|$ being odd by Lemma \ref{lem:lem1}. Note that $x=gZ\in S$ is a real element of odd order and $|C_S(x)|$ is odd. By using \cite{GAP}, we may assume that $q\ge 13$.
Note that $\PSL_3(2)\cong\PSL_2(7)$.

Assume  $q$ is even. As in the proof Lemma 4.4 in \cite{GNT}, $G=\SL_3(q)$ contains a real element $g$ of order $q+1$ with $|C_G(g)|=q^2-1$. Clearly, $q^2-1$ is odd and the result follows.

Assume  $q$ is odd. By Corollary 6.7 in \cite{TZ04}, the regular unipotent element $u$ of $G$ is rational and hence it is real with $|C_G(u)|=dq^2$ being odd. We can take $g=u$.

(3) Assume that  $q=2^f$ is even and $n\ge 4$. Assume that $n\ge 4$ is even and $(n,q)\neq (6,2), (7,2)$. Then $G=\SL_n(q)$ contains an element $g$ of order $\ell$, where $\ell$ is a ppd of $2^{nf}-1$. As $g$ can be embedded into  $\Sp_n(q)\leq \SL_n(q)$, $g$ is real. Now by Lemma 2.4 in \cite{MT} and \cite{BPS}, $|C_G(g)|=(q^n-1)/(q-1)$  is odd.

Assume that $n\ge 5$ is odd. Then $n-1$ is even and thus as in the previous case, $\SL_n(q)$ contains a real element $g$ of order $\ell$, where $\ell$ is a ppd of $2^{(n-1)f}-1$ and by Lemma 2.4 \cite{MT}, $|C_G(g)|=q^{n-1}-1$ is odd.
The cases $\LL_6(2)$ and $\LL_7(2)$ can be checked using \cite{GAP}.

(4) Assume $q$ is odd and that $n\ge 4$. Assume first that $n\ge 5$ is odd, or $n/d$ is even or $q\equiv 1$ mod $4$. Then  by Corollary 6.7 in \cite{TZ04}, a regular unipotent element $u$ is real (indeed, it is rational). Note that $x=uZ(G)\in S\unlhd G^*$ is a regular unipotent element in $G^*$ and by \cite[Proposition 5.1.9]{Carter},  $|C_{G^*}(x)|=q^{n-1}$ and hence $|C_S(x)|$  is odd.

We are left with the cases $n$ is even, $q\equiv 3$ mod $4$ and $n/d$ is odd.
As in the proof of  \cite[Lemma 4.4]{GNT}, $G=\SL_n(q)$ has a real unipotent element $g$  with $|C_G(g)|=q^n(q-1)$ and $\gcd(|g|,|Z(G)|)=1.$   By Lemma \ref{lem:lem1}, it suffices to show that $|C_G(g):Z(G)|=|C_G(g)|/d$ is odd.
Since $q\equiv 3$ mod $4$, we have $q-1\equiv 2$ mod $4$. As $n$ is even and $d=\gcd(n,q-1),$ we deduce that $d\equiv 2$ mod $4$. It follows that  $(q-1)/d$ is odd and so $|C_G(g):Z(G)|$ is odd as wanted.
The proof is now complete.

\end{proof}

We next consider the  unitary groups. 
\begin{lem}\label{lem:U}
Proposition \ref{prop:Lie type}  holds for $\UU_n(q)$, where $q=p^f$ and $n\ge 3.$
\end{lem}

\begin{proof} Let  $d=|Z(\SU_n(q))|=\gcd(n,q+1)$, $S=\SU_n(q)/Z(\SU_n(q))$ and $G^*=\PGU_n(q)$ so $[G^*,G^*]=S$. 

(1) Assume  $n=3$. Then $d=\gcd(3,q+1) $ is odd and $q\ge 3$. Assume first assume $q\ge 16$. 
Assume $q=2^f$ is even.
Let $\ell$ be a ppd of $2^{2f}-1$ so $\ell\mid 2^f+1$. Embed $\SL_2(q)\cong \SU_2(q)$ into $\SU_3(q)$  and let $g$ be a real element of order $\ell$ in $\SL_2(q)$. Then $|C_G(g)|=(q+1)^2$ (see, for example, \cite{SF}). Since $(q+1)^2$ is odd, the lemma follows by Lemma \ref{lem:lem1}. 
Assume that $q$ is odd. By Corollary 6.7 in \cite{TZ04}, the regular unipotent element $u$ of $G$ is rational and hence $u$ is real with $|C_G(u)|=dq^2.$ Since $dq^2$ is odd, the result follows. The cases with $3\leq q<16$ can be found in Table \ref{tab:Lie-type}.

(2) Assume that  $q=2^f$ is even,  $n\ge 4$ and $(n,q)\neq (6,2),(7,2).$ Let $m=2\lfloor n/2\rfloor$. Then there is a ppd $\ell>2$ of $2^{mf}-1$ and a real element $g\in \Sp_m(q)<G$ of order $\ell$ with 

\[|C_G(g)|=\begin{cases}
(q^n-1)/(q+1),\quad\quad\, n\equiv 0\, (\text{mod}\,4)\\
(q^{n/2}+1)^2/(q+1),\quad   n\equiv 2\, (\text{mod}\, 4)\\
q^{n-1}-1,\qquad\quad\quad\quad\,\,\,   n\equiv 1\, (\text{mod} \,4)\\
(q^{(n-1)/2}+1)^2,\qquad\quad   n\equiv 3\, (\text{mod} \,4).
\end{cases}\]
For the proof, see \cite[Lemma 3.5]{DNT}. Clearly $|C_G(g)|$ is odd in these cases.
The cases $\UU_6(2)$ and $\UU_7(2)$  are listed in Table \ref{tab:Lie-type}.

(3) Assume $q$ is odd and $n\ge 4$. If $n\ge 5$ is odd or $n/d$ is even or $q\equiv 1$ mod $4$, then by arguing as in part (4) in  the proof of the previous lemma, $G^*$ has a regular unipotent element $u\in S$ which is rational and hence is real with $|C_{G^*}(u)|=q^{n-1}$  and so $|C_S(u)|$ is odd.


Thus we assume that $n\ge 4$ is even, $q\equiv 3$ (mod $4$) and $n/d$ is odd. 

Assume that $n\equiv 0$ (mod 4). Let $\ell$ be a ppd of $p^{nf}-1$. Since $n$ is even, we can embed $\Sp_n(q)$ in $G$ and so there exists an element $g\in \Sp_n(q)$ of order $\ell.$ By \cite[Proposition 3.1]{TZ04}, $g$ is real in $\Sp_n(q)<G$. By Lemma 3.5 \cite{DNT}, we have  $|C_G(g)|=(q^n-1)/(q+1)$.

By Lemma \ref{lem:lem1}, it suffices to show that $|C_G(g):Z(G)|=|C_G(g)|/d$ is odd since $|g|=\ell>nf\ge n$ so $\gcd(|g|,|Z(G)|)=1.$ Recall that $d=\gcd(n,q+1).$ Since $n/d$ is odd, we have $(n)_2=(d)_2$. Moreover, as $q\equiv 3$ mod 4, $(q-1)_2=2$.

 By \cite[Lemma A.4]{BG}, we have $$(q^n-1)_2=(q^2-1)_2(n/2)_2=2(q+1)_2(n/2)_2=(q+1)_2(n)_2$$ and \[(d(q+1))_2=(d)_2(q+1)_2=(q+1)_2(n)_2.\]
Thus $(q^n-1)/(d(q+1))$ is odd.

Next, assume that $n\equiv 2$ mod 4. Write $n=2m$. Then $m\ge 3$ is odd.  Let $\ell$ be a ppd of $p^{mf}-1$ and let $g\in H:=\Sp_{2m}(q)<G$ be an element of order $\ell$. Note that $g$ lies in a unique maximal torus  $T$ of order $(q^n-1)/(q+1)$ of  $G$ by considering the order formulas for the maximal tori of $G.$   
It follows that $|C_G(g)|=(q^n-1)/(q+1)$ and $g$ is real in $G$ (see also \cite[Table B.4]{BG}). The same argument as in the previous case implies that $(q^n-1)/(d(q+1))$ is odd. The lemma is proved.
\end{proof}

\begin{lem}\label{lem:BC}
Proposition \ref{prop:Lie type}  holds for $\PSp_{2n}(q)$\emph{ ($n\ge 2$)} and $\Omega_{2n+1}(q)$\emph{ ($n\ge 3$)}, where $q=p^f$.
\end{lem}

\begin{proof} Let $G=\Sp_{2n}(q)$ or $\Spin_{2n+1}(q)$ and let $d=|Z(G)|=\gcd(2,q-1)$. Then $S=G/Z(G)$. Note that if $S=\Omega_{2n+1}(q)$, then $q$ is odd and if $(n,q)=(2,2)$, then $S=\PSp_4(2)'$. By \cite[Proposition 3.1]{TZ05}, any semisimple element of $G$ is real and by \cite[Theorem 1.9]{TZ04} (or \cite[Lemma 2.2]{TZ05}) every unipotent element of $G$ is real if $q$ is a square and $p>2.$ 
By using GAP \cite{GAP}, we may assume that $S$ it not $\PSp_6(2)$ nor $\PSp_4(2)'$. Then a ppd $\ell_1$ of $p^{2nf}-1$ exists.
 By \cite[Lemma 2.4]{TZ04}, $G$ has a semisimple  element $g_1$ of order $\ell_1$  and $|C_G(g)|=q^n+1$. Note that $g_1$ is real and $\gcd(|g_1|,|Z(G)|)=1$. 

\medskip
(1) Assume $n=2$. We claim that $|C_G(g_1):Z(G)|=(q^2+1)/d$ is odd.

If $q$ is even, then $q=2^f\ge 4$ and $d=1$. So $G=S$  and $|C_G(g_1)|=2^{2f}+1$ is odd. 
Assume  that $q$ is odd. Then $d=2$.  By \cite[Lemma A.4]{BG}, $(q^2+1)_2=2$ and thus $|C_G(g_1):Z(G)|=(q^2+1)/d$ is odd. The claim above is proved  and the lemma follows from Lemma \ref{lem:lem1}.

(2) Assume $n\ge 3$ and $q$ is even. Then $S=\PSp_{2n}(q)$. In this case, $|C_G(g_1)|=q^n+1$ is odd and we are done.

(3) Assume $n\ge 3$ and $q$ is odd. Then $d=2$. If $n$ is even, then $(q^n+1)_2=2$ and so $|C_G(g_1):Z(G)|$ is odd and the lemma follows as in  (1). So we may assume that $n\ge 3$ is odd. If $q\equiv 1$ mod 4, then a regular unipotent element $u\in G$ is real  and $|C_G(u):Z(G)|$ is a power of $q$ by \cite[Proposition 5.1.9]{Carter}. Hence we can assume that $q\equiv 3$ (mod $4$).

Let $\ell_2$ be a ppd of $p^{nf}-1$. By \cite[Lemma 5.5]{MNO}, $G$ contains an element $g_2$ of order $\ell_2$ with $|C_G(g_2)|=q^n-1$. Now $g_2$ is semisimple and hence is real with $\gcd(|g_2|,|Z(G)|)=1$. By \cite[Lemma A.4]{BG}, $(q^n-1)_2=(q-1)_2=2$ and thus $|C_G(g_2):Z(G)|=(q^n-1)/2$ is odd. 
\end{proof}

\begin{lem}\label{lem:D}
Proposition \ref{prop:Lie type}  holds for $\PP\Omega^\epsilon_{2n}(q)$, where $n\ge 4,\epsilon=\pm$ and $q=p^f$.
\end{lem}

\begin{proof} Let $G=\Spin_{2n}^\epsilon(q)$ and let $d=|Z(G)|$. Then $d=\gcd(2,q-1)^2$ if $\epsilon=+,n$ is even; and $d=\gcd(4,q^n-\epsilon)$ otherwise. We have $S=G/Z(G)$. 

(1) Assume $p=2$. Obviously, $d=1$ and so $G=S$. By using \cite{GAP}, we can assume that $(n,q)\neq (4,2),(5,2)$. Assume that $n$ is even. By the proof of Lemma 4.4 in \cite{GNT}, if $(n,q,\epsilon)\neq (4,4,+)$, then $G$ has a semisimple real element $s_1$ of odd order with $|C_G(s_1)|=(q^{n-1}-\epsilon)(q-1)$ and if $(n,q,\epsilon)= (4,4,+)$, then there is a real element $s_1\in G$ order $13$ with $|C_G(s_1|=3(4^3-1)$. Similarly, for odd $n$, $G$ has a real semisimple element $s_1$ of order $p_1$, a ppd of $2^{(2n-2)f}-1$, with $|C_G(s_1)|=(q^{n-1}+1)(q+\epsilon)$ .   Clearly $|C_G(s_1)|$ is odd in all cases and we are done.  

(2) Assume $p>2$.  Note that $d=4$ if $n$ is even and $\epsilon=+$; and $d=\gcd(4,q^n-\epsilon)$ otherwise. If $q\equiv 1$ mod $4$, then a regular unipotent element $u\in G$ is real by \cite[Lemma 2.2]{TZ05} and since $|C_G(u):Z(G)|$ is a power of $q$, we are done by Lemma \ref{lem:lem1}. So, we may assume that $q\equiv 3$ mod $4$. Hence $(q-1)_2=2$.

(a) Assume that $n$ is even. By the proof of \cite[Lemma 4.4]{GNT}, $G$ has a real element $s_1$ of odd order with $|C_G(s_1)|=(q^{n-1}-1)(q-1)$ if $\epsilon=+$ and a real element $s_2$ of odd order with $|C_G(s)_2|=q^n+1$ if $\epsilon=-$.
By \cite[Lemma A.4]{BG}, we have $(q^{n-1}-1)_2=(q-1)_2$ and thus $|C_G(s_1):Z(G)|=(q^{n-1}-1)(q-1)/4$ is odd. Similarly $(q^n+1)_2=2$ and $d=(4,q^n+1)=2$ whence $|C_G(s_2):Z(G)|=(q^n+1)/d$ is also odd.

(b) Assume that $n$ is odd. Assume first that $\epsilon=+$. Then $n(q-1)/2$ is odd as $q\equiv 3$ mod $4$. Hence by \cite[Theorem 1.9]{TZ04}, a regular unipotent element $u\in G$ is rational and hence is real. Moreover, $|C_G(u):Z(G)|$ is a power of $q$ by \cite[Proposition 5.1.9]{Carter} and hence it is odd.

Next, assume that $\epsilon=-.$ As in the proof of \cite[Lemma 4.4]{GNT}, $G$ has a semisimple real element $s_1$ with $|C_G(s_1)|=(q^{n-1}+1)(q-1).$ By \cite[Lemma A.4]{BG}, we have $(q^{n-1}+1)_2=2$ and $(q-1)_2=2$. Moreover, $(q^n+1)_2=(q+1)_2$ is divisible by $4$, so $d=\gcd(4,q^n+1)=4$. Thus $|C_G(s_1):Z(G)|=(q^{n-1}+1)(q-1)/d$ is odd. The proof is now complete.
\end{proof}

{\scriptsize
\begin{table}[h]
\renewcommand\thetable{C}
\[
\begin{array}{lcr|lcr|lcr} \hline
S & x & |C_{S}(x)| &S &x &  |C_{S}(x)|&S &x &  |C_{S}(x)| \\ \hline

\PSL_2(7) & 3A& 3&\PSL_2(11)&5A&5&\PSL_3(3)   &3B&9\\

\PSL_3(4) & 5A& 5&\PSL_3(5)&5B&25&\PSL_3(7)   &7B&49\\

\PSL_3(8) & 3A& 63&\PSL_3(9)&3B&81&\PSL_3(11)   &11B&121\\

\PSL_6(2) & 9A& 63&\PSL_7(2)&9A&63&\PSU_3(3)   &3B&9\\

\PSU_3(4) & 5E& 25&\PSU_3(5)&5B&25&\PSU_3(7)   &7B&49\\

\PSU_3(8) & 9A& 27&\PSU_3(9)&3B&81&\PSU_3(11)   &11B&121\\

\PSU_3(13) & 13B& 169&\PSU_6(2)&7A&7&\PSU_7(2)   &7A&63\\

\PSp_4(2)' & 5A& 5&\PSp_6(2)&7A&7&\PP\Omega^+_{8}(2)   &7A&7\\

\PP\Omega^-_{8}(2) & 7A& 21&\PP\Omega^+_{10}(2) &17A&51&\PP\Omega^-_{10}(2)   &17A&17\\

\hline
\end{array}
\]

\caption{The triples $(S,x,|C_S(x)|)$  in Lemmas \ref{lem:L} - \ref{lem:D} with $S$ a finite simple group of Lie type.}
\label{tab:Lie-type}
\end{table}}

Proposition  \ref{prop:Lie type} now follows from Lemmas \ref{lem:exceptional} - \ref{lem:D} and Theorem \ref{th:simple} follows from Lemma \ref{lem:non-lie} and Proposition   \ref{prop:Lie type}.

\begin{proof}[\textbf{Proof of Theorem \ref{th:almost simple}}] Let $G$ be an almost simple group with simple socle $S$. 
Assume first that $S\not\in\mathcal{L}$. By Theorem \ref{th:simple} (1), $S$ has a nontrivial real $p$-element $x$ for some odd prime $p$ such that $|C_S(x)|$ is odd. Thus there exists a $2$-element $y\in S$ such that $x^y=x^{-1}$ by \cite[Lemma 2.1]{GNT}.  Hence $y$ normalizes $\la x\ra$ and $\la x,y\ra$ is solvable. Since $x^2\neq 1$, $y$ must be nontrivial. We next claim that $[x,y^g]\neq 1$ for all $g\in G.$ By way of contradiction, assume that $[x,y^g]=1$ for some $g\in G.$ It follows that $y^g\in C_S(x)$ as $y\in S\unlhd G$. But this is impossible as $y^g$ is a nontrivial $2$-element but $|C_S(x)|$ is odd. Therefore, no $G$-conjugate of $y$ commutes with $x$.

Finally, if $S\in\mathcal{L}$, then the result follows immediately from Theorem \ref{th:simple} (2) since $y$ normalizes $\la x\ra$ so that $\la x,y\ra$ is solvable. The proof is complete.
\end{proof}

\section{Solvability and Nilpotency criteria}\label{sec3}
We prove Theorem \ref{th:solvable} in this section. For brevity, let $(*)$ be the property: 

`For every pair of distinct primes $p$ and $q$ and for every pair of elements $x,y\in G$ with $x$ a $p$-element and $y$ a $q$-element, if $\la x,y\ra$ is solvable, then $\la x,y^g\ra$ is solvable for all $g\in G$.' The main part of Theorem \ref{th:solvable} is to show that if $G$ satisfies (*), then $G$ is solvable. 

\begin{proof}[\textbf{Proof of Theorem \ref{th:solvable}}]

Clearly, if $G$ is solvable, then $\la x,y^g\ra$ is solvable for all $x,y,g\in G$. We will now prove the converse. That is, if $G$ satisfies  $(*)$, then $G$ is solvable. Clearly, every proper subgroup of $G$ satisfies  $(*)$ and hence by induction, every proper subgroup of $G$ is solvable.

\medskip
(1) Assume that $G$ is not a nonabelian simple group. Let $N$ be a minimal normal subgroup of $G$. Then $N$ is solvable by the induction hypothesis. We claim that $\overline{G}=G/N$ satisfies (*) and hence by induction again, $G/N$ solvable and so $G$ is solvable. To see this, let $p$ and $q$ be distinct primes and let $xN,yN$ be a $p$-element and a $q$-element, respectively, such that $\la xN,yN\ra$ is solvable. Replacing $x$ and $y$ by their powers, we may assume that $x$ is a $p$-element and $y $ is a $q$-element in $G$. Since $\la x,y\ra N/N=\la xN,yN\ra$ is solvable and $N$ is solvable, $\la x,y\ra$ is solvable.  It follows that $\la x,y^g\ra$ is solvable for all $g\in G$ and thus $\la xN,(yN)^{gN}\ra$ is solvable for all $gN\in G/N.$ Hence $G/N$ satisfies  (*).

\medskip
(2) Assume $G$ is a nonabelian simple group. Then $G$ is a minimal simple group. By \cite[Corollary~1]{Thompson}, $G$ is isomorphic to one of the following simple groups:
\begin{enumerate}
\item $\PSL_2(2^r)$, $r$ is a prime.
\item $\PSL_2(3^r)$, $r$ is an odd prime.
\item $\PSL_2(r)$, $r>3$ is a prime such that $5\mid r^2+1.$
\item ${}^2\textrm{B}_2(2^r)$, $r$ is an odd prime.
\item $\PSL_3(3).$
\end{enumerate}
We consider each case separately.  Recall that if $q\ge 4$ is a prime power and $d=\gcd(2,q-1)$, then $G=\PSL_2(q)$ has a maximal subgroup $H$ which is isomorphic to $ \textrm{D}_{q+1}$ if $q$ is odd and $\textrm{D}_{2(q+1)}$ if $q$ is even. In both cases, $H=\la a,y\ra$, where $a$ is an element of order $(q+1)/d$, $y$ is an involution and $a^y=a^{-1}$. 

Write $q=s^r$, where $s\in \{2,3\}$ and $r$ is a prime with $r>2$ if $s=3.$  First, assume that a ppd $p$ of $s^{2r}-1$ exists. Let $x$ be a generator of a Sylow  $p$-subgroup of $H$. Then $\la x\ra$ is a Sylow $p$-subgroup of $G$, $x\in \la a\ra$ and $x^y=x^{-1}$. So $\la x,y\ra$ is solvable and $H=N_G(\la x\ra)$.  By inspecting Tables $8.1-8.2$  in \cite{BHR} or Table B in \cite{BBGH}, $H$ is the unique maximal subgroup of $G$ containing $\la x\ra$. Now let $g\in G$  be such that $y^g\not\in H.$ It then follows that $\la x,y^g\ra=G$, which is nonsolvable.

For the cases when a ppd of $s^{2r}-1$ does not exist, we have $s=2$ and $r=3$. In this case, $G=\PSL_2(8)$. Let $x\in G$ be a real element of order $9$ and $y\in G$  an involution inverting $x$. It is easy to check that $\la x,y^g\ra=G$ for some $g\in G.$

Next, consider the case $G=\PSL_2(r)$, where $r>3$ is an odd prime and $5\mid r^2+1$. Let  $x$ be an element of order $r$ in $G$. 
Let $t$ be a prime divisor of $(r-1)/2$. Note that $\la x\ra$ is a Sylow $r$-subgroup of $G$ and $H=N_G(\la x\ra)$ is the unique  maximal subgroup of $G$ containing $x$. (See \cite{BHR} or \cite[Table B]{BBGH}). Now $|H|=r(r-1)/2$. Let $y$ be an element of order $t$. Then $\la x,y\ra$ is solvable. Using the same argument as above, we can find $g\in G$ such that $\la x,y^g\ra$ is nonsolvable.

Assume that $G\cong {}^2\textrm{B}_2(2^r)$, where $r$ is an odd prime.  Let $p$ be a prime divisor of $2^r+\sqrt{2^{r+1}}+1$ and let $x$ be a generator of a Sylow $p$-subgroup of $G$. Then $H=N_G(\la x\ra)\cong (2^r+\sqrt{2^{r+1}}+1):4$ is the unique maximal subgroup of $G$ containing $\la x\ra$, where $\la x\ra$ is a Sylow $p$-subgroup of $G$ (see \cite[Table C]{BBGH} or by inspecting \cite[Table 8.16]{BHR}). A similar argument as above excludes this case.

Finally, for $G\cong\PSL_3(3)$, by using GAP \cite{GAP}, $G$ has elements $x$ of order $13$ and $y$ of order $3$ normalizing $\la x\ra$. Moreover, $H=\la x,y\ra$ is the unique maximal subgroup of $G$ containing $x$. So, there exists $g\in G$ such that $\la x,y^g\ra=G$. The proof is now complete.
\end{proof}

Using Theorem  \ref{th:almost simple}, we deduce the following solvability criterion.

\begin{prop}\label{prop:solvable}
Let $G$ be a finite group. Assume that for every pair of two nontrivial elements $x$ and $y$, where $x$ is a $p$-element for some odd prime $p$, and $y$ is a $2$-element,  if $\la x,y\ra$ is solvable, then $x$ commutes with some $G$-conjugate of $y$. Then $G$ is solvable.
\end{prop}

\begin{proof}
Let $G$ be a counterexample to the proposition with minimal order. Then $G$ is not solvable.

\smallskip
(1) We first claim that the solvable radical of $G$ is trivial. Assume by contradiction that the solvable radical, say $K$, of $G$ is nontrivial.  Let $xK,yK\in G/K$, where $xK$ is a $p$-element for some odd prime $p$ and $yK$ is a $2$-element. Replacing $x$ and $y$ by their powers, we may  assume that $x$ and $y$ are $p$-element and $2$-element, respectively. Let $A=\la x,y\ra$. Assume that $AK/K=\la xK,yK\ra$ is solvable. Then $AK/K\cong A/A\cap K$ is solvable and since $K$ is solvable, $A\cap K$ is also solvable, whence $A$ is solvable. By the hypothesis on $G$, $x$ and $y^g$ commute for some $g\in G$ and hence $xK$ and $(yK)^{gK}$ commute in $G/K$. Therefore, $G/K$ satisfies the hypothesis of the proposition and since $|G/K|<|G|$, by the minimality of $|G|$, $G/K$ is solvable and so $G$ is solvable, a contradiction.

\smallskip
(2) Let $N$ be a minimal normal subgroup of $G$. Then $N\cong S^k$ for some nonabelian simple group $S$ and $k\ge 1.$ By Theorem  \ref{th:almost simple}, there exist a $p$-element $a$ and a $2$-element $b$ in $S$ such that $\la a,b\ra$ is solvable but $a$ does not commute with any $\Aut(S)$-conjugate of $b$, where $p$ is an odd prime. 
Let $x=(a,a,\dots,a)$ and $y=(b,b,\dots,b)$ be two elements in the diagonal subgroup of $N.$ Note that $x$ is a $p$-element and $y$ is a $2$-element in $N$. Since $\la a,b\ra$ is solvable, we see that $\la x,y\ra$ is solvable. By the hypothesis, there exists an element $g\in G$ such that $x$ commutes with $gyg^{-1}.$

\smallskip
(3) Let $C=C_G(N)$. Then $C\unlhd G$ and $G/C$ has a unique minimal normal subgroup $NC/C\cong N.$ To simplify the notation, we may assume that $C=1$.  Since $G$ can be embedded into $\Aut(N)\cong \Aut(S)\wr \Sym_k$ and $g\in G,$ we can write $g=(u_1,u_2,\dots,u_k)\sigma$, where $u_i\in\Aut(S)$ and $\sigma\in\Sym_k.$ Now we can check that $$gyg^{-1}=(u_1,u_2,\dots,u_k)\sigma (b,b,\dots,b)\sigma^{-1} (u_1,u_2,\dots,u_k)^{-1}=(b^{u_1^{-1}},b^{u_2^{-1}},\dots,b^{u_k^{-1}}).$$
Since $x$ and $gyg^{-1}$ commute, $a$ and $u_1bu_1^{-1}$ commute, which is a contradiction as $u_1\in\Aut(S)$. This completes the proof.
\end{proof}


We are now ready to prove Theorems \ref{th:nilpotent} and \ref{th:direct-factor}.

\begin{proof}[\textbf{Proof of Theorem \ref{th:nilpotent}}]
Clearly, the `only if' direction of the theorem holds trivially. Now assume that for every pair of distinct primes $p$ and $q$ and for every pair of elements $x,y\in G$ with $x$ a $p$-element and $y$ a $q$-element, if $\la x,y\ra$ is solvable, then $x,y^g$ commute for some $g\in G.$ By  Proposition \ref{prop:solvable}, $G$ is solvable and so by \cite[Corollary E]{DGHP}, $G$ is nilpotent.
\end{proof}

\begin{proof}[\textbf{Proof of Theorem \ref{th:direct-factor}}]
Assume that for every pair of elements $x,y\in G$ with $x$ a $p$-element and $y$ a $2$-element, where $p>2$ is a prime, if $\la x,y\ra$ is solvable, then $x,y^g$ commute for some $g\in G.$ By Proposition  \ref{prop:solvable}, $G$ is solvable. Therefore, for every $p$-element $x$ with $p$ an odd prime and every $2$-element $y$,  $x,y^g$ commute for some $g\in G.$

Let $P$ be a Sylow $2$-subgroup of $G$. Let $1\neq z\in Z(P)$. Then $P\leq C_G(z)\leq G.$ If $C_G(z)=G$, then $z\in Z(G)$ and since $G/\la z\ra$ satisfies the hypothesis of the theorem, $P/\la z\ra$ is a direct factor of $G/\la z\ra$. So $G/\la z\ra=P/\la z\ra\times B/\la z\ra$, where $B$ is a normal subgroup of $G$ containing $z$ such that $B\cap P=\la z\ra$ and $|B:\la z\ra|$ is odd. By Schur-Zassenhaus theorem $B=A\times \la z\ra$ for some normal subgroup $A$ of odd order of $B$. It then implies that $G=P\times A.$ 

Assume that $C_G(z)<G.$ Since $|G:C_G(z)|$ is odd, $G$ contains a $p$-element $x$, where $p$ is an odd prime, such that $x^G\cap C_G(z)=\emptyset$ by \cite[Theorem 1]{FKS}. However, this contradicts the claim in the first paragraph.
\end{proof}

\section{Applications to graphs on groups}\label{sec4}
We prove all the corollaries  in this section.

\begin{proof}[\textbf{Proof of Corollary \ref{cor:SCC-graph}}]
It suffices to show that if the expanded $\SCC$-graph of $G$ coincides with the solvable graph of $G$, then $G$ is solvable.
Let $x,y\in G$ be non-conjugate elements in $G$.  Assume that $x$ and $y$ are joined in the expanded $\SCC$-graph. Then $\la x,y^g\ra$ is solvable for some $g\in G$. By the hypothesis, it follows that $x$ and $y$ are joined in the solvable graph of $G$ and thus $\la x,y\ra$ is solvable. Therefore, if $\la x,y^g\ra$ is solvable for some $g\in G,$ then $\la x,y\ra$ is solvable for all non-conjugate elements $x,y\in G.$

We now claim that if $\la x,y\ra$ is solvable for some non-conjugate elements $x,y\in G,$ then $\la x,y^g\ra$ is solvable for all $g\in G.$ Indeed, assume that $x,y\in G$ are non-conjugate in $G$ such that $\la x,y\ra$ is solvable. Let $g\in G$ and let $z=y^g$. Now $x$ and $z$ are joined in the expanded $\SCC$-graph of $G$ since $\la x,z^{g^{-1}}\ra=\la x,y\ra$ is solvable, whence $\la x,z\ra$ is solvable by the hypothesis. The claim is proved. Finally, the corollary follows from Theorem \ref{th:solvable}.
\end{proof}

\begin{rem}
There is another proof of Corollary \ref{cor:SCC-graph} using results in \cite[Corollary 1]{Guest10} and \cite[Corollary 1.5]{GGKP09} stating that a finite group $G$ is solvable if and only if for all $x,g\in G$, $\la x,x^g\ra$ is solvable. Too see this, assume that the expanded $\SCC$-graph coincides with the solvable graph of a finite group $G$. Arguing as in the previous proof, if $x\in G$ and $g\in G$, then by definition, $x$ and $x^g$ are adjacent in the expanded $\SCC$-graph of $G$ so that $\la x,x^g\ra$ is solvable. Hence $G$ is solvable by the aforementioned results.
\end{rem}

\begin{proof}[\textbf{Proof of Corollary \ref{cor:NCC-graph}}]  If $G$ is nilpotent, then clearly the expanded $\SCC$-graph is equal to the expanded $\NCC$-graph. For the converse, assume that the expanded $\SCC$-graph and the expanded $\NCC$-graph coincide. Then for all $x,y\in G$ such that $x$ and $y$ are not conjugate in $G$, if $\la x,y^g\ra$ is solvable for some $g\in G$, then $\la x,y^{gh}\ra$ is nilpotent for some $h\in G.$ Replace $y$ by $y^g,$ we see that if $x,y\in G$ are not conjugate in $G$ and $\la x,y\ra$ is solvable, then $\la x,y^g\ra$ is nilpotent for some $g\in G.$ Note that if $x$ and $y$ have coprime orders, then so do $x$ and $y^g$ and thus if $\la x,y^g\ra$ is nilpotent, then $x$ and $y^g$ commute. Therefore, we can apply Theorem \ref{th:nilpotent} to conclude that $G$ is nilpotent.
\end{proof}

\begin{rem}
In view of these results, we could study the expanded $\SCC$-graph ($\NCC$-graph or $\CCC$-graph) by forgetting the edges among the same conjugacy classes. That is, we say that there is an edge between two elements $x,y\in G$ in the $\SCC$-graph ($\NCC$-graph or $\CCC$-graph) if and only if $x$ and $y$ are not conjugate in $G$ and $\la x,y^g\ra$ is solvable (nilpotent or abelian) for some $g\in G.$
\end{rem}

\begin{proof}[\textbf{Proof of Corollary \ref{cor:invariable-gen-graph}}]  Let $G$ be a finite non-nilpotent group.
If $G$ is minimal non-nilpotent, then the result is clear. So we will prove the other direction. Assume that for every pair of distinct primes $p$ and $q$ and for every pair of elements $x,y\in G$ with $x$ a $p$-element and $y$ a $q$-element, if $\la x,y\ra\neq G$, then $[x,y^g]=1$ for some element $g\in G.$

We first claim that $G$ is solvable. Suppose by contradiction that $G$ is not solvable. Let $p$ be an odd prime, let $x\in G$ be a $p$-element and $y\in G$ be a $2$-element such that $\la x,y\ra$ is solvable. Since $G$ is not solvable, $\la x,y\ra\neq G$ and thus by the hypothesis, $[x,y^g]=1$ for some $g\in G$. Now by Proposition \ref{prop:solvable}, $G$ is solvable, which is a contradiction.

We now show that $G$ is minimal non-nilpotent by induction on $|G|$. Note that $G$ is non-nilpotent by the hypothesis.
Let $N$ be a minimal normal subgroup of $G$. Since $G$ is solvable, $N$ is an elementary abelian $p$-group for some prime $p$ dividing $|G|.$ Let $P$ be a Sylow $p$-subgroup of $G$. Then $N\cap Z(P)>1$ and let $z\in N\cap Z(P)$ with $|z|=p.$ Let $C=C_G(z)$. 

Assume that $C=G$. Then $z\in Z(G)$ and $G/\la z\ra$  satisfies the hypothesis of the corollary. So by the induction hypothesis, $G/\la z\ra$ is a minimal non-nilpotent group. (Note that if $G/\la z\ra$ is nilpotent, then so is $G$.) Hence if $M/\la z\ra$ be a maximal subgroup of $G/\la z\ra$, then $M/\la z\ra$ is nilpotent and since $z\in Z(M)$, $M$ is nilpotent. Now let $M$ be a maximal subgroup of $G$. If $z\in M$, then $M$ is nilpotent by the previous claim. Assume that $z\not\in M$. Then $G=M\la z\ra$ by the maximality of $M.$ Let $a,b\in M$ with coprime prime power order. Then $\la a,b\ra\leq M$ and thus $[a,b^g]=1$ for some $g\in G.$ Since $G=\la z\ra M$ and $z\in Z(G)$, we have $b^g=b^m$ for some $m\in M$, and so $[a,b^m]=1$ for some $m\in M$. By Corollary E in \cite{DGHP}, $M$ is nilpotent. Thus $G$ is minimal non-nilpotent in this case.

Now assume that $C<G.$ By \cite[Theorem 1]{FKS}, there exists a $q$-element $y\in G$ with $y^G\cap C=\emptyset.$ Since $P\leq C,$ we have $q\neq p.$ If $\la z,y\ra\neq G$, then $[z,y^g]=1$ for some $g\in G$, which contradicts the choice of $y$. 
Thus $G=\la z,y\ra$ and hence $G=N\la y\ra$. It follows that $N=P$ is a normal Sylow $p$-subgroup of $G$ and $Q=\la y\ra$ is a cyclic Sylow $q$-subgroup of $G$. Since $N$ is a minimal normal subgroup of $G$ and is abelian, $Q$ is a maximal subgroup of $G$ and acts irreducibly on $N$.

Let $M$ be a maximal subgroup of $G$. We claim that $M$ is nilpotent. If $N\nleq M$, then $G=MN$ and so $M\cap N=1$ since $N$ is abelian. By comparing orders, we see that $|M|=|Q|$, hence $M$ is nilpotent. So, assume that $N\leq M$. If $M=N$, then $M$ is nilpotent. Thus we assume that $M>N$ and so $M=N(M\cap Q)$ with $1<Y=M\cap Q\leq \la y\ra$. Therefore, $Y=\la y_1\ra$ for some $1\neq y_1\in \la y\ra$. Observe that $y_1^G=y_1^N$ as $y_1\in \la y\ra$ and $G=\la y\ra N.$ Let $a\in N$. Then $\la a,y_1\ra\neq G$ and so $[a,y_1^g]=1$ for some $g\in G$. As $y_1^g=y_1^m$ for some $m\in N$,  $[a,y_1^m]=1$. Since $a,m\in N$, $am=ma$ and so $[a,y_1]=1$. Hence $y_1$ centralizes $N$ as $a\in N$ is chosen arbitrarily. Therefore, $M=N\la y_1\ra$ is nilpotent.
\end{proof}

In the following, we will make use of Theorem B in \cite{DGHP}. 
\begin{proof}[\textbf{Proof of Corollary \ref{cor:invariable-gen-SCC}}] 
Observe that any quotient of $G$ satisfies the same hypothesis, so we may assume that $R(G)=1.$ We now prove that $G$ is a simple group. Let $N$ be a minimal normal subgroup of $G$.  Then $N\cong S_1\times S_2\times\dots\times S_k\cong S^k$, where $S=S_1$ is a nonabelian simple group and $k\ge 1.$
If $G=N$, then $k=1$ and thus $G$ is simple. So, assume that $N\neq G.$ By Theorem B in \cite{DGHP}, there  exist two distinct prime divisors $p$ and $q$ of $|S|$ such that for all $x_1,y_1\in S$ with $|x_1|=p$ and $|y_1|=q$, $\la x_1,y_1\ra$ is nonsolvable. Now let $x=(x_1,x_1,\dots,x_1)\in N$ and $y=(y_1,y_1,\dots,y_1)\in N$. Then $\la x,y\ra\leq N$ and so $\la x,y\ra\neq G$, therefore, $\la x,y^g\ra$ is solvable for some $g\in G.$ However, the projection of $\la x,y^g\ra$ to the first component of $N$ is the subgroup $\la x_1,u_1y_1{u_1^{-1}}\ra$ for some $u_1\in \Aut(S)$. (See the proof of Proposition \ref{prop:solvable}.) Sine $|u_1y_1{u_1^{-1}}|=|y_1|=q$, we know that $\la x_1,u_1y_1{u_1^{-1}}\ra$ is nonsolvable, which is a contradiction.
\end{proof}

\begin{rem} The converse of this corollary holds if $G/R(G)$ is a minimal simple groups. However,
it is not true that if a finite simple group $S$ satisfies the hypothesis of Corollary \ref{cor:invariable-gen-SCC}, then it is a minimal simple group. For a counterexample, we can take $S=\PSL_2(31)$. The only nonsolvable subgroups of $S$ are either isomorphic to $\Alt_5$ or $S$ itself. We can check using \cite{GAP} that for every pair of nontrivial elements $x,y\in S$, either $x$ and $y$ invariably generate $S$ or $\la x,y^g\ra$ is solvable for some $g\in G.$
\end{rem}

However, for sporadic simple groups, the Tits group and the alternating groups, we obtain the following. Note that $\Alt_5\cong \PSL_2(4)\cong\PSL_2(5)$ is a minimal simple group.
\begin{lem}\label{lem:inva-gen}
Let $S$ be a sporadic simple group, the Tits group or an alternating group of degree $n\ge5$. If $S$ satisfies the hypothesis of Corollary \ref{cor:invariable-gen-SCC}, then $S\cong \Alt_5.$
\end{lem}

\begin{proof} Assume that $S$ satisfies the hypothesis of Corollary \ref{cor:invariable-gen-SCC}.

(1) If $S$ is a sporadic simple group or the Tits group, then we can find two distinct prime dividsors $p$ and $q$ of $|S|$ such that the subgroup $\la x,y\ra$ is nonsolvable for any $x,y\in S$ with $|x|=p$ and $|y|=q.$ The primes $p,q$ and the possibilities  for $\la x,y\ra$ (excluding $S$ itself) are listed in Table \ref{tab:inva-gen}.

(2) Assume that $S=\Alt_n$ with $n\ge 5.$ If $n=5$, then $\Alt_5$ is a minimal simple group. So assume that $n\ge 6$ and let $m=n-1\ge 5.$
As in the proof of Proposition 3.3 in \cite{DGHP}, there exist two distinct primes $p,q$ with $m/2\leq p< q\leq m.$ We claim that $\Alt_n$ together with two primes $p$ and $q$ satisfy all the hypotheses of \cite[Lemma 3.2]{DGHP} and so $\la x,y\ra$ is nonsolvable for all $x,y\in \Alt_n$ with $|x|=p$ and $|y|=q.$ Since $p<q\leq n-1$, it is possible to find $x,y\in \Alt_{n-1}$ and so $\la x,y\ra\neq \Alt_n.$ Therefore, $\Alt_n$ with $n\ge 6$, does not satisfy the hypothesis of the corollary.

(i) Since $q>p\ge (n-1)/2>2$, $p$ odd and so $q\ge p+2\ge (n-1)/2+2>n/2$. Thus a Sylow $q$-subgroup of $\Alt_n$ is of order $q$. Furthermore, as $p\ge (n-1)/2$, we see that $3p>n$ so that a Sylow $p$-subgroup of $\Alt_n$ has order $p$ or $p^2$. 

(ii) If $p>(n-1)/2$, then $2p>n-1\ge q$. On the other hand, if $p=(n-1)/2$, then $n-1$ is even so $q<n-1=2p.$ Thus in either cases,  $p\nmid q-1$

(iii) Since $q\ge p+2$, we see that $q\nmid (p^2-1).$

(iv) Next, since $q+p\ge (p+2)+p=2p+2\ge n-1+2=n+1>n$, $\Alt_n$ has no element of order $pq$.

Thus $\la x,y\ra$ is nonsolvable. The proof is now complete.
\end{proof}

Using the argument as in the proof of Theorem B in \cite{DGHP}, for the finite simple groups of Lie type $S$, one may only need to look at the cases when $S$ has small Lie rank to obtain the classification of all the examples in Corollary \ref{cor:invariable-gen-SCC}.
{\scriptsize
\begin{table}[h]
\renewcommand\thetable{D}
\[
\begin{array}{llll|llll} \hline
S & p&q & \la x,y\ra &S &p &q &\la x,y\ra \\ \hline

{\rm M}_{11} & 2&11 & \LL_2(11) &{\rm M}_{12} &2 &11 &\LL_2(11), {\rm M}_{11} \\ 

{\rm M}_{22} & 2&11 & \LL_2(11) & {\rm M}_{23} &2 &11 &\LL_2(11), {\rm M}_{11}, {\rm M}_{22} \\ 

{\rm M}_{24} & 2&23 & \LL_2(23) & {\rm J}_1 &3 &11 &\LL_2(11) \\ 

{\rm J}_2& 2& 7 & \LL_2(7),\UU_3(3) & {\rm J}_3 & 2 &19 &\LL_2(19) \\ 

{\rm J}_4& 5& 37 & \UU_3(11) & {\rm HS} & 2 &11 &\LL_2(11), {\rm M}_{11}, {\rm M}_{22} \\ 

{\rm McL}& 2& 11 & \LL_2(11), {\rm M}_{11}, {\rm M}_{22} & {\rm He} & 5 &17 &\LL_2(16),\PSp_4(4) \\ 

{\rm Suz}& 7& 13 & \LL_2(13), {\rm G}_2(4) & {\rm Ly} & 7 &31 &{\rm G}_2(5)   \\ 

{\rm Ru}& 5& 29 & \LL_2(29) & {\rm O'N} & 5 & 19 & {\rm J}_1   \\ 

{\rm HN}& 7& 19 & \UU_3(8) & {\rm Th} & 5 &19 &\LL_2(19)   \\ 

{\rm Co}_1& 11& 13 & 3\cdot {\rm Suz} & {\rm Co}_2 & 2 & 23 & {\rm M}_{23}   \\ 

{\rm Co}_3& 2& 23 & {\rm M}_{23} & {\rm Fi}_{22} & 7 &13 &\LL_2(13),\OO_7(3)   \\ 

{\rm Fi}_{23}& 11& 13 & 2\cdot {\rm M}_{22} & {\rm Fi}_{24}' & 17 &23 &{\rm Fi}_{23}   \\ 

{\rm B}& 19& 31 & {\rm Th} & {\rm M} & 2 &59 &\LL_2(59)   \\ 

{}^2{\rm F}_4(2)'& 5& 13 & \LL_2(25) &  & &&  \\ \hline

\end{array}
\]
\caption{The quadruples $(S,p,q,\la x,y\ra)$ in Lemma \ref{lem:inva-gen} with $S$ a sporadic simple group or the Tits group}
\label{tab:inva-gen}
\end{table}}

\section{A characterization of the solvable radicals}\label{sec5}

For a finite group $G$, recall that the Baer-Suzuki theorem states that for any element $x\in G$, $\la x^G\ra$ is nilpotent if and only if $\la x,x^g\ra$ is nilpotent for all $g\in G.$ This gives a characterization of the Fitting subgroup $F(G) $ of $G$, the largest normal nilpotent subgroup of $G$. Generalizing this to solvable radical, for an odd prime $p\ge 5$ and  an element $x\in G$ of order $p$, it is shown in \cite{Guest10, GGKP09} that $x\in R(G)$ if and only if $\la x,x^g\ra$ is solvable for all $g\in G.$ In \cite{GL14}, Guest and Levy showed that for an arbitrary element $x\in G$, $x\in R(G)$ if and only if for all odd primes $p$ and all $p$-elements $y\in G$, the subgroup $\la x,y\ra$ is solvable. We next prove the following conditional characterization of the solvable radical as follows. 

\begin{thm}\label{th:radical} 
Let $G$ be a finite group, let $p\ge 5$ be a prime and let $x\in G$ be an element of order $p$. Assume that for all primes $r\neq p$, and all $r$-elements $y\in G$, if $\la x,y\ra$ is solvable, then $\la x,y^g\ra$ is solvable for all $g\in G$.  Then $x\in R(G)$.
\end{thm}

In the next proposition, we give a reduction of Theorem \ref{th:radical} to almost simple groups and eliminate the sporadic simple groups and the alternating groups. Recall that the Fitting subgroup of $G$, denoted by $F(G),$ is the largest nilpotent normal subgroup of $G$, the layer $E(G)$ is a subgroup generated by all components of $G$, where $L$ is a component of $G$ if $L$ is a quasisimple subnormal subgroup of $G$ and finally, the generalized Fitting subgroup of $G$, denoted by $F^*(G)$, is defined by $F^*(G)=F(G)E(G)$. We write $\pi(G)$ for the set of all distinct prime divisor of $|G|.$ We include the case $|x|=3$ in the reduction result below.
 
\begin{prop}\label{prop:conj1_reduction} 
Let the pair $(G,x)$ be a counterexample to Theorem \ref{th:radical}  with $|G|$ minimal. Then 

\begin{enumerate}
\item[${\rm (i)}$] $G$ is an almost simple with simple socle $S$ and $G=S\la x\ra.$
\item[${\rm (ii)}$] If $H$ is a proper subgroup of $G$ containing $x$ then $x\in R(H)$.
\item[${\rm (iii)}$] $S$ is a finite simple group of Lie type.
\end{enumerate}

\end{prop}

\begin{proof}
Let the pair $(G,x)$ be a counterexample to Theorem \ref{th:radical}  with $|G|$ minimal. Then $x$ is an element of prime order $p>2$ and $x\not\in R(G)$. Moreover, for all primes $r\neq p$ and for all $r$-elements $y\in G$, if $\la x,y\ra$ is solvable, then $\la x,y^g\ra$ is solvable for all $g\in G$. We call the latter property $(**)$ for brevity.  

\smallskip
(1) Observe first that if $H$ is any proper subgroup of $G$ containing $x$, then the pair $(H,x)$ also satisfies property (**). Thus by the minimality  of $|G|$, we have $x\in R(H)$, proving (ii).

\smallskip
(2) We claim  that $G=\la x^G\ra$. Let $N=\la x^G\ra\unlhd G.$ Assume by contradiction that $N\neq G.$ By (1) above, $x\in R(N)$. Since $R(N)$ is characteristic in $N$ and $N\unlhd G,$ we see that $R(N)\unlhd G$, hence $x\in R(N)\subseteq R(G)$, which is a contradiction.

\smallskip
(3) $R(G)=1$. Let $N=R(G)$ and assume that $N>1.$  Since $x\not\in N$, $xN$ has order $p$ in $G/N$. Let $r\neq p$ be a prime and let $yN\in G/N$ be an $r$-element such that $\la xN,yN\ra$ is solvable. Replace $y$ by its power, we may assume that $y\in G$ is an $r$-element. By the hypothesis, $\la xN,yN\ra=\la x,y\ra N/N$ is solvable and since $N$ is solvable, we deduce that $\la x,y\ra$ is solvable. Thus $\la x,y^g\ra$ is solvable for all $g\in G$ and hence $\la xN,(yN)^{gN}\ra$ is solvable for all $gN\in G/N$. This implies that $(G/N,xN)$ satisfies property $(**)$ and thus $xN\in R(G/N)$. However, $R(G/N)=1$ as $N=R(G)$, we obtain a contradiction. Thus $R(G)=1$.

\smallskip
(4) $G=N\la x\ra$, where $N$ is a unique  minimal normal subgroup of $G$ and it is nonabelian. 
Since $R(G)=1$, we know that $F^*(G)=E(G)$ is a direct product of nonabelian minimal normal subgroups of $G$. Let $N$ be such a minimal normal subgroup. Then $A:=N\la x\ra$ is a subgroup of $G$ containing $x$.  Assume that $A\neq G.$ By (1), $x\in R(A)$ by the minimality of $G$. However, as $R(A)\cap N=1$, $[R(A),N]=1$ and thus $x\in R(A)\leq C_G(N)\unlhd G$ which forces $G=C_G(N)$ or $N\leq Z(G)$, which is impossible. Therefore $G=N\la x\ra$ as wanted.

Assume by contradiction that $G$ has another minimal normal subgroup $N_1\neq N.$  By (3), $N_1$ is nonabelian and $N_1\cap N=1$ and so $[N,N_1]=1.$ By the claim above, we also have $G=N_1\la x\ra$. If $x\in N$, then $G=N$ is a simple group. Thus $N=N_1$. Hence we assume that $x$ does not lie in $N$ nor $N_1$. In particular, $N\cap \la x\ra=1=N_1\cap \la x\ra$. Since $G=N\la x\ra=N_1\la x\ra$, we have $|N|=|N_1|$. Now $NN_1$ is a normal subgroup of $G$ so that $|NN_1|=|N|^2$ divides $|G|=p|N|$, which is impossible. 

\smallskip
(5) $G$ is an almost simple group with a simple socle $S$.
Since $N$ is a minimal normal nonabelian subgroup of $G$, $N\cong S^k$ for some simple group $S$ and $k\ge 1.$ If $k=1$, then $G$ is almost simple and we are done. So, assume $k\ge 2.$

In this case, $G$ embeds into $\Aut(N)\cong \Aut(S)\wr \Sym_k$. Replace $x$ and $G$ by their conjugates in $\Aut(N)$, we can write $x=(\rho,1,\dots,1)\tau$, where $\rho\in\Aut(S)$ and $\tau=(1,2,\dots,k)$ is a $k$-cycle. We can check that $x^p=(\rho,\rho,\dots,\rho)$.  As $|x|=p$, we have $\rho=1$ and thus $x=\tau.$

Let $r\neq p$ be an odd prime. Assume that there exists a nontrivial $r$-element  $ a\in S$ such that $\la a,a^t\ra$ is nonsolvable for some $t\in S.$ 

Let $y=(a,a,\dots,a)\in N.$ Then $y$ is an  $r$-element and $y^x=y$ so $\la x,y\ra$ is solvable.
Let $g=(t,1,1,\dots,1)\in N$. By the hypothesis, $\la x,y^g\ra$ is solvable. Observe that $y^g=(a^t,a,\dots,a)$ and $(y^g)^x=(a,a^t,a,\dots,a)$. Therefore, the projection of $\la y^g,(y^g)^x\ra$ to the first component of $N$ contains $\la a^t,a\ra$ which is not solvable. Hence $\la y^g,(y^g)^x\ra$ is not solvable and so $\la x,y^g\ra$ is not solvable as well. This contradiction proves our claim.

To finish the proof, we need to show that there exists an odd prime $r\neq p$ and an $r$-element $a\in S$ such that $\la a,a^t\ra$ is nonsolvable for some $t\in S.$

Now if $|S|$ has a  prime divisor $r$ with $5\leq r\neq p$, then the claim follows from \cite[Theorem $A^*$]{Guest10}. On the other hand, if no such primes $r$ exists, then $\pi(S)=\{2,3,p\}$. From the classification of finite simple groups with three distinct prime divisors (see, for example, \cite{Herzog}), $S$ is isomorphic to one of the following groups: $$\Alt_5,\Alt_6,\PSp_4(3)\cong\PSU_4(2), \PSL_2(7),\PSL_2(8),\PSU_3(3),\PSL_3(3),\PSL_2(17).$$
By using \cite{GAP}, we can always find an element $a\in S$ of order $3$ and $t\in S$ such that $\la a,a^t\ra$ is nonsolvable.

\smallskip
(6) $S$ is not  a sporadic simple group nor the Tits group. Assume that $S$ is a sporadic simple group or the Tits group. Since $|\Aut(S):S|\leq 2$ and $p$ is odd, we see that $G=S$.  Using the character tables of these simple groups as well as their maximal subgroups available in \cite{GAP}, for each odd prime $p$ and each  element $x$ of prime order $p$, we can show that there exists a maximal subgroup $H$ of $S$ containing $x$ but $x\not\in R(H)$, violating (1), except for the following pairs $(S,p):$ \[({\rm M}_{23},23), ({\rm Fi}_{24}',29), ({\rm J}_{4},43),({\rm J}_{4},29),({\rm Ly},67), ({\rm Ly},37),({\rm B},47),({\rm J}_{1},19),({\rm J}_{1},7).\]   
Except for the last pair, $x$ lies in a unique maximal subgroup $H=N_S(\la x\ra)$, which is a Frobenius group with kernel $\la x\ra$ and a cyclic complement $R$. Note that $R$ is not a $2$-group and so let $r$ be a prime divisor of $|R| $. Then $2<r\neq p$. Let $y$ be a nontrivial $r$-element in $R$. Then clearly $\la x,y\ra\leq H$ is solvable. Let $g\in G-H$. Then clearly $\la x,y^g\ra=S$, which is nonsolvable. For the remaining case, $N_S(\la x\ra)\cong C_7:C_6$ and we can find an element $y\in S$ of order $3$ and $g\in G$ such that $\la x,y^g\ra$ is nonsolvable while $\la x,y\ra$ is solvable.

\smallskip
(7) $S$ is not an alternating groups $\Alt_n$ with $n\ge 5.$ Suppose by contradiction that $S$ is an alternating group acting on a set $\Omega$ of size $n\ge 5.$ Since $|\Aut(S):S|$ is a power of $2$, we have that $G=S.$ We can use MAGMA \cite{magma} to eliminate the cases $5\le n\leq 24.$ We assume that $n\ge 25.$  

(7a) Assume first that $n=p$ is a prime. Then $n\ge 29$ and $x$ is a $p$-cycle. Moreover, $\la x\ra$ is a Sylow $p$-subgroup of $S$. As in the proof of Proposition 4.2 in \cite{BBGH}, if $p$ is  of the form $(q^d-1)/(q-1)$ for some prime power $q$ and some prime $d\ge 2$, then $x\in \PSL_d(q)$. Since $p\ge 29$, the pairs $(d,q)$ cannot be $(2,2)$ nor $(2,3)$ and thus $\PSL_d(q)$ is nonsolvable. Now assume that $p$ is not of that form, then $x$ lies in a unique maximal subgroup $H=N_S(\la x\ra)\cong C_p:C_{(p-1)/2}$. Note that $(p-1)/2$ cannot be a power of $2$ as otherwise, $p$ would have the form $(2^{2f}-1)/(2^f-1)$. Therefore, we can find an odd prime $r\neq p$ dividing $(p-1)/2$ and an element $y\in  H$ of order $r$. Clearly $\la x,y\ra\leq H$ is solvable. By choosing $g\in S$  $y^g\not\in H$, then $\la x,y^g\ra=S$ is nonsolvable.

(7b) Assume that $n$ is not a prime. Since $|x|=p\ge 3,$ $x$ is a product of $k$ disjoint cycles of length $p$, for some integer $k\ge 1$. If $kp<n$, then $x$ fixes some point $\alpha\in\Omega$ and so it lies in $H$, a point stabilizer of $\alpha$ in  $S$, so $H\cong \Alt_{n-1}$, which is a simple group and thus $x\in R(H)=1$, a contradiction. Hence $n=kp$. Since $n$ is not a prime, $k\ge 2.$ In this case, $x$ lies in the subgroup $H\leq \Alt_n$ which stabilizes a subset $\Gamma\subseteq \Omega$ of size $p$, so $x\in H=(\Sym_{n-p}\times \Sym_p)\cap \Alt_n.$  If $p=3$, then $R(H)$ is a cyclic group of order $3$ generated by a $3$-cycle and so cannot contain $x$. If $p\ge 5,$ then $n-p=(k-1)p\ge 5$ and thus $R(H)=1.$
\end{proof}

We now consider the finite groups of Lie type. We will follow the notation and definitions from \cite{GLS}. Assume that $G$ is an almost simple group with socle $G_0$, a finite simple group of Lie type defined over a field $\FF_q$ of size $q=r^f$, where $r$ is a prime, and $x\in G$ is an element of order $p\ge 5$ such that $G=G_0\la x\ra$. Write $G_0^*$ for the group of inner-diagonal automorphisms of $G_0.$
We follow closely the proofs of Theorem $A^*$ in \cite{Guest10} and Theorem 1.3  in \cite{Guest12}. In fact, these two papers already contain all the information needed to rule out these cases.

To handle the small Lie rank groups, we may need a couple of additional results. The next lemma is basically Lemmas 6.1 and 6.2 in \cite{King}.

\begin{lem}\label{lem:conjugates}
Let $G$ be a finite group and let $p$ be a prime divisor of $|G|$. Assume that $G$ has a cyclic Sylow $p$-subgroup, say $P$. Let $1\neq x\in P$. Let $M$ be a maximal subgroup of $G$ containing $x$.

\begin{enumerate}
\item[$(1)$] If $x$ lies in  $M$ and $M^g$ for some $g\in G$, then $mg\in N_G(\la x\ra)$ for some $m\in G.$
\item[$(2)$] The number of $G$-conjugates of $M$ containing $x$ is $|N_G(\la x\ra)|/|N_M(\la x\ra)|.$
\end{enumerate}
\end{lem}

The next lemma is \cite[Lemma 2.2]{GS}.
\begin{lem}\label{lem:2 parabolic}
Let $G_0$ be a simple group of Lie type and  suppose that $G_0\unlhd G\leq G_0^*$ and let $x\in G.$
\begin{itemize}
\item[$(a)$] If $x$ is unipotent, let $P_1$ and $P_2$ be distinct maximal parabolic subgroups of $G$ containing a common Borel subgroup of $G$ with unipotent radicals $U_1$ and $U_2$. Then $x$ is conjugate to an element of $P_i-U_i$ for some $i\in\{1,2\}.$
\item[$(b)$] If $x$ is semisimple, assume that $x$ lies in a parabolic subgroup of $G$. If the rank of $G_0$ is at least $2$, then there exists a maximal parabolic subgroup $P$ with a Levi complement $J$ such that $x$ is conjugate to an element of $J$ not centralized by any 
(possibly solvable) Levi component of $J$.

\end{itemize}
\end{lem}

We first eliminate the cases when $G_0\cong \PSL_2(q)$ or ${}^2{\rm B}_2(q)$.
\begin{lem}\label{lem:PSL2}
The almost simple groups $G$ with socle $\PSL_2(q)$ where $q=r^f$, $r$ is a prime and $f\ge 1$, is not a minimal counterexample to Theorem \ref{th:radical}.
\end{lem}

\begin{proof}
Let $G_0=\PSL_2(q)$ be the socle of $G$ and let $x\in G$ be an element of order $p\ge 5$ such that $(G,x)$ satisfies the hypothesis of Theorem \ref{th:radical}  with $|G|$ minimal. Then $(G,x)$ satisfies the conclusion of Proposition \ref{prop:conj1_reduction}. 
Since $|x|=p$ is odd, either $x$ is a field automorphism of $G_0$ or $x\in G_0.$

(a) Assume $x$ is a field automorphism. Let $C=C_{G_0}(x)$. Then $C\cong \PSL_2(q_0)$. Hence there exists an involution $i\in C$ such that $\la x,i\ra$ is solvable. Note that $G$ has only one class of involutions since $G_0$ has only one class of involutions and $|G:G_0|$ is odd. By  Lemma 3.3 \cite{Guest12}, there exists an involution $j\in G_0$ such that $\la x,j\ra$ is not solvable. However, as $i$ and $j$ are $G_0$-conjugate, we obtain a contradiction.

(b) $x\in G_0$ and hence $G=G_0.$ We can assume $q\ge 7$ and $q\neq 9.$

 If $\la x,i\ra$ is solvable for some involution $i\in G$, then we are done by invoking \cite[Lemma 3.3]{Guest12} again as $G_0$ has one class of involutions. Otherwise, $p\mid q$ and so $x$ is a transvection and thus we can assume $x\in \PSL_2(p)$. Since $p\ge 5$ and $(G,x)$ is a minimal counterexample, we deduce that $G=\PSL_2(p)$ and $x$ lies in  the Borel subgroup $B =UT$, where $U=\la x\ra$ is a Sylow $p$-subgroup of $G$ and $T$ is cyclic of order $(p-1)/2$.  Since $\la x,i\ra$ is not solvable for all involutions $i\in G,$ we must have that $p\equiv 3$ mod 4. Now by \cite{BBGH} or \cite{BHR} and Lemma \ref{lem:conjugates}, $B$ is the unique maximal subgroup of $G$ containing $x$. Thus if $r$ is any prime divisor of $|T|$ and $y\in T$ is an element of order $r$, then $\la x,y\ra$ is solvable but $\la x,y^g\ra$ is not solvable for any $g\in G$ with $y^g\not\in B.$
\end{proof}

\begin{lem}\label{lem:Sz}
The almost simple groups $G$ with socle $G_0={}^2{\rm B}_2(q)$, where $q=2^f$ and $f\ge 3$ is odd, is not a minimal counterexample to Theorem \ref{th:radical}.
\end{lem}

\begin{proof}
Note that $G_0$ and hence $G$ has only one class of involution and either $x$ is a field automorphism or $x\in G_0=G$. 
 If $x$ is a field automorphism, then $C_{G_0}(x)\cong {}^2{\rm B}_2(2^{f/p})$ and $C_{G_0}(x)$ contains an involution $i$ such that $\la x,i\ra$ is solvable. If $x\in G_0$, then $G=G_0$ and a similar claim holds by inspecting the maximal subgroups of $G_0$ in \cite[Table 8.16]{BHR} (note that $x$ does not lie in a simple subfield subgroup $ {}^2{\rm B}_2(2^e) $ with $3\leq e\mid f$ and $e<f$ by the minimality of $(G,x)$). Hence the lemma follows by using Lemma 3.12 in \cite{Guest12}.
\end{proof}

We next exclude the cases when $x$ is an outer automorphism of $G_0$ which is not an inner-diagonal automorphism. The argument is similar to Lemma 7 \cite{Guest10}.

\begin{lem}\label{lem:Outer}
Assume that $G_0$ is not isomorphic to ${}^2{\rm B}_2(2^f)$ with $f\ge 3$ odd nor $\PSL_2(q)$, and $x\in G-G_0$  is not an inner-diagonal automorphism of $G_0$, then $(G,x)$ is not a minimal counterexample to Theorem \ref{th:radical}.
\end{lem}

\begin{proof}
Since $p\ge 5$ and $x\not\in G_0^*,$ $x$ is a field automorphism.
Since the untwisted Lie rank of $G_0$ is at least $2$, by \cite[7.2]{GL}, we may assume that $x$ is a standard field automorphism. By \cite[Theorem 3.2.8]{GLS}, $x$ acts nontrivially as a field automorphism on a fundamental $\SL_2$-subgroup. However, this contradicts the minimality of the pair $(G,x)$.
\end{proof}

The next lemma is a modified version of  \cite[Lemma 3.2]{Guest12}. We use the notation $\GL_n^+(q)$ for $\GL_n(q)$ and $\GL_n^-(q)$ for $\textrm{GU}_n(q)$. The same convention applies to $\PSL_n^\epsilon(q)$ and $\PGL_n^\epsilon(q)$ for $\epsilon=\pm.$
\begin{lem}\label{lem:lift}
If $x\in G\leq \PGL_n^\epsilon(q)$ does not lift to an element of order $p$ in $\GL^\epsilon_n(q)$, then $(G,x)$ is not a minimal counterexample to Theorem \ref{th:radical}.

\end{lem}

\begin{proof} Let $(G,x)$ be a minimal counterexample to Theorem \ref{th:radical}  and assume that $x$ does not lift to an element of order $p$ in $\GL_n^\epsilon(q)$.
As in the proof of part (a) of Lemma 3.2 in \cite{Guest12}, we have $p\mid (q-\epsilon,n)$ and we can assume that $n=p$ and $x$ acts irreducibly on the natural module $V$. Moreover, a Sylow $p$-subgroup of $\GL_n^\epsilon(q)$ is contained in a subgroup of type $(q-\epsilon)\wr \Sym_p$ and $x$ is nontrivial in $\Sym_p.$
Since $p\ge 5,$ $x$ cannot lie in the solvable radical of the subgroup of type $(q-\epsilon)\wr \Sym_p$ above.
\end{proof}

\begin{lem}\label{lem:unipotent}
If $x\in G$ is unipotent, then $(G,x)$ is not a minimal counterexample to Theorem \ref{th:radical}.
\end{lem}

\begin{proof}
Since $x\in G$ is unipotent and $p\ge 5$, we can assume that $G=G_0$ is not isomorphic to ${}^2{\rm B}_2(q)$, ${}^2{\rm F}_4(q)$, ${}^2{\rm G}_2(q)$ and $\PSL_2(q)$. So the untwisted Lie rank of $G$ is at least $2$.

Assume first that $G\not\cong \PSU_3(q)$ nor ${}^3{\rm D}_4(q)$.
As in the proof of \cite[Lemma 11]{Guest10}, by using Lemma \ref{lem:2 parabolic} and Table 3 in \cite{Guest10}, there exists a parabolic subgroup $P$ with unipotent radical $U$ of $G$ such that $x$ is contained in $P-U$ and acts nontrivially on a Levi complement  of $P$ and since $p\ge 5$, this Levi complement is almost simple. But then $(G,x)$ is not a minimal counterexample.

Now assume that $G=\PSU_3(q)$. Since $5\leq p\mid q$, we have $q\ge 5.$ If $x$ is a transvection, then $x$ lies in a subgroup of type $\textrm{GU}_2(q)$. Thus we assume that $x$ is not a transvection and thus $x$ is a regular unipotent element. Since $(G,x)$ is a minimal counterexample, the only maximal subgroups of $G$ could contain $x$ are the maximal parabolic subgroups. It follows that $x$ lies in a unique maximal parabolic subgroup of type $q^{1+2}:(q^2-1)$ (see \cite[Tables 8.5 and 8.6]{BHR}). So we can find an involution $i$ such that $\la x,i\ra$ is solvable. Now choose another involution $j$ that does not lie in this parabolic  subgroup, then $\la x,j\ra$ is not solvable. Finally, the result follows sine $G$ has only one class of involutions.

Now assume that $G={}^3{\rm D}_4(q)$. Since $5\leq p\mid q$, we have $q\ge 5.$ By Lemma \ref{lem:2 parabolic}, we may assume that $x$ acts nontrivially on a Levi component of type $\SL_2(q)$ or $\SL_2(q^3)$ of a parabolic subgroup. Hence $(G,x)$ is not a minimal counterexample.
\end{proof}

\begin{lem}\label{lem:Ree}
The almost simple groups $G$ with socle $G_0={}^2{\rm G}_2(q)$, where $q=3^f$ and $f\ge 3$ is odd, is not a minimal counterexample to Theorem \ref{th:radical}.
\end{lem}

\begin{proof}
Since $|x|=p\ge 5$, by Lemma \ref{lem:Outer}, $x$ is semisimple and $G=G_0$. Table 8.43 in \cite{BHR} lists all maximal subgroups of $G$. Since $(G,x)$ is minimal counterexample, $x$ cannot lie in the maximal subgroups of type $2\times \PSL_2(q)$ or the subfield subgroups. The remaining maximal subgroups are solvable and have even order, so we can find an involution $i$ such that $\la x,i\ra$ is solvable and by Theorem 1.3 \cite{Guest12}, there exists an involution $j$ such that $\la x,j\ra$ is nonsolvable. Now the result follows sine $G$ has only one class of involutions.
\end{proof}

\begin{lem}\label{lem:G2}
The almost simple groups $G$ with socle $G_0\cong{\rm G}_2(q)', {}^3{\rm D}_4(q)$, or ${}^2{\rm F}_4(q)$ are not minimal counterexamples to Theorem \ref{th:radical}.
\end{lem}

\begin{proof}
By Lemmas \ref{lem:Outer} and \ref{lem:unipotent}, $x$ is semisimple and $G=G_0.$

(1) Assume $G_0={\rm G}_2(q)'$. Since ${\rm G}_2(2)'\cong\PSU_3(3)$, we can assume that $q\ge 3.$ By \cite[p. 546]{GS} or \cite[Lemma 3.9]{Guest12}, $x$ normalizes but does not centralizes a subgroup of type $\SL^\epsilon_3(q)$. So $(G,x)$ is not a minimal counterexample.

(2) Assume $G_0\cong  {}^3{\rm D}_4(q)$. The cases when $q=2,3$ can be handled by MAGMA \cite{magma}, we can assume that $q\ge 5.$
If $x$ lies in a maximal parabolic subgroup, then we get a contradiction by using Lemma \ref{lem:2 parabolic}. So we assume that $x$ does not lie in any maximal parabolic subgroups. It follows that $C_G(x)$ is a maximal torus.  As in the proof of Lemma 3.9 in \cite{Guest12},  if $p\mid q^2-q+1$, then $x$ is contained in a subgroup of type $(q^2-q+1)\circ \SU_3(q).$ If $x$ does not centralize $\SU_3(q)$, then $(G,x)$ is not a minimal counterexample. If $x$ centralizes $\SU_3(q)$, then $x$ centralizes some unipotent elements and hence is contained in a maximal parabolic subgroup. The case when $p\mid q^2+q+1$ can be handled similarly. Now assume $p\mid q^4-q^2+1$. By \cite[Table 8.51]{BHR}, $x$ lies in a unique maximal subgroup $H$ of the form $(q^4-q^2+1):4$. So we can find an involution $i$ such that $\la x,i\ra$ is solvable and by the uniqueness of $H$, we can find $g\in G$ such that $i^g\not\in H$ and hence $\la x,i^g\ra$ is not solvable.

(3) Assume $G_0\cong  {}^2{\rm F}_4(q)$ with $q=2^{2f+1}$ and $f\ge 1$. If $x$ lies in a maximal parabolic subgroup, then we get a contradiction by using Lemma \ref{lem:2 parabolic}. So we assume that $x$ does not lie in any maximal parabolic subgroups. As in the previous case, by using Table 7 in \cite{Guest12} and \cite{Malle}, we deduce that either $(G,x)$ is not a minimal counterexample or $x$ lies in a unique maximal subgroup which is solvable and has even order (it is the normalizer of a Sylow $p$-subgroup of $G$). We can obtain a contradiction as before.
\end{proof}

\begin{lem}\label{lem:F4}
The almost simple groups $G$ with socle $G_0\cong {\rm F}_4(q), {\rm E}_6(q), {\rm E}_7(q), {\rm E}_8(q)$, or ${}^2{\rm E}_6(q)$ are not minimal counterexamples to Theorem \ref{th:radical}.
\end{lem}

\begin{proof}
By Lemmas \ref{lem:Outer} and \ref{lem:unipotent}, $x$ is semisimple and $G=G_0.$ As in the proof of the previous lemma or \cite[Lemma 3.8]{Guest12} and \cite[\S 12--\S14]{Guest10}, $x$ does not lie in any maximal parabolic subgroups of $G$ and thus $C_G(x)$ is a maximal torus. Using the information from Tables 2-5 in \cite{Guest12} and also Tables 2--3 in \cite{BBGH}, where some maximal subgroups of $G$ that contain a Sylow $p$-subgroup are listed. In these cases, either $(G,x)$ is not a minimal counterexample or $x$ lies in a unqiue maximal subgroup $H=N_G(P)$ which is also a normalizer of a maximal torus containing $P$ with $x\in P.$ We know that $P\neq H$ by \cite[Theorem 1.1]{GMN} since $p\ge 5.$ Thus we can find an $r$-element $y\in H$, where  $r\neq p$ is a prime such that $\la x,y\ra$ is solvable. By choosing $g\in G$ with $y^g\not\in H$, we see that $\la x,y^g\ra$ is not solvable.
\end{proof}

\begin{proof}[\textbf{Proof of Theorem \ref{th:radical}}]
Let the pair $(G,x)$ be a minimal counterexample to Theorem \ref{th:radical} with $|G|$ minimal. Then $x\in G$ is an element of prime order $p$ and $\la x,y^g\ra$ is solvable for all $g\in G$ whenever $y$ is an $r$-element for some prime $r\neq p$ and $\la x,y\ra$ is solvable.
By Proposition \ref{prop:conj1_reduction}, $G=S\la x\ra$ is an almost simple group with simple socle $S$, where $S$ is a finite simple group of Lie type defined over a finite field of size $q$, where $q$ is a prime power. Moreover, if $H$ is any proper overgroup of $x$, then $x\in R(H)$. 

By Lemmas \ref{lem:PSL2}, \ref{lem:Sz}, \ref{lem:Ree}, \ref{lem:G2} and \ref{lem:F4}, we can assume that $G_0$ is a classical groups and $G_0\not\cong \PSL_2(q)$. 
By Lemmas \ref{lem:Outer} and \ref{lem:unipotent}, we have $G_0\unlhd G\leq G_0^*$ and $x$ is semisimple.  
Since $p\ge 5,$ either $G_0\cong\PSL_d^\epsilon(q)$ with $n\ge 3$ or $G=G_0.$

Inspecting the proof of Lemma 3.6 in \cite{Guest12} (see also Sections $7-11$ in \cite{Guest10}), where we use Lemma \ref{lem:lift} rather than \cite[Lemm 3.2]{Guest12}, the only cases when $(G,x)$ could be a minimal counterexample  are when $G_0\cong\PSL_d^\epsilon(q)$,  $p$ is a ppd of $q^d-1$ or $q^{2d}-1$ if $\epsilon=-$, $d$ is an odd prime and $x$ acts irreducibly on the natural module for $G_0$. Moreover, since $p>d$, $x$ is contained in a unique maximal subgroup $H$ of type $\GL_1^\epsilon(q^d)\cdot d$. Using the same argument as in the proof of the previous lemma, we obtain a contradiction. This completes the proof of the theorem.
\end{proof}

\begin{rem} We record here some remarks on Theorem \ref{th:radical}.
\begin{enumerate}
\item Theorem \ref{th:radical} does not hold if $x$ is an involution. For example, if $G=\mathrm{S}_5$ and $x$ is any transposition, then $\la x,y\ra$ is solvable for all $3$-elements $y\in G$ while $\la x,y\ra$ is non-solvable for all $5$-elements $y\in G.$

\item If we insist that $r\neq p$ is an odd prime, the conclusion does not hold either. We can take $G=\PSL_3(3)$ and $x$ a transvection of order $3$. Then  for any odd prime $r\neq 3$ (so $r=13$), we can check that $\la x,y\ra$ is not solvable for any $r$-elements $y\in G$. However, there is an element $z\in G$ of order $4$ or $8$ such that $\la x,z\ra$ is solvable but $\la x,z^g\ra$ is not solvable for some $g\in G.$

\item Theorem \ref{th:radical} might hold for elements of order 3 as well but the proof is more involved. Note that the reduction to almost simple groups works for the case $|x|=3$.
\item We believe that this theorem holds for $p$-elements for any odd primes $p$. However, we are not able to reduce this to almost simple groups yet.
\item The following problem is parallel to Theorem 1 in \cite{GL14}: ``Let $G$ be a finite group and let $x\in G$. Then $x\in R(G)$ if and only if for all odd primes  $p$, and all $p$-elements $y\in G$, if $\la x,y\ra$ is solvable, then $\la x,y^g\ra$ is solvable for all $g\in G$.''  If $x$ is an element of prime order $p\ge 5$, then  $\la x,x^g\ra$ is solvable for all $g\in G$ since $\la x,x\ra$ is solvable. In this case, $x\in R(G)$ by the main results in \cite{GGKP09,Guest10}. 
\item In view of Proposition \ref{prop:conj1_reduction}, it is an interesting problem to determine all finite nonsolvable groups $G$ which contains an element $x$ such that $x$ lies in the solvable radical of every proper overgroup of $x$ in $G$. This would be an extension of the classification of minimal simple groups due to Thompson.
\end{enumerate}
\end{rem}

Problem 2.3 in \cite{BCNS} asks for a description of the set of vertices of the expanded $\SCC$-graph of a finite group $G$ which are joined to all other vertices. We restate this problem as follows.

\begin{problem}\label{prob1}
Let $G$ be a finite group. Classify all elements $x\in G$ such that for each conjugacy class $C$ of $G$, there exists $y\in C$ such that $\la x,y\ra$ is solvable.
\end{problem}

As already noted in \cite{BCNS}, all the involutions in $\LL_2(2^f)$ and $\textrm{J}_1$ satisfy the assumption of Problem \ref{prob1}. For other examples, one can take $G=\LL_2(8)\cdot 3$ and $x$ an outer automorphism of order $3$ or $G=\LL_2(2^5)\cdot 5$ and $x$ is an outer automorphism of order $5.$ Thus the elements that satisfy the hypothesis of this problem may not be in the solvable radical of the group.
For dominant vertices in the expanded $\NCC$-graph, we obtain the following result that is an easy consequence of the main theorem in \cite{FM} due to Fumagalli and Malle  generalizing  Wielandt's theorem on subnormal subgroups. Recall that if $p$ is a prime, then $O_p(G)$ is the largest normal $p$-subgroup of $G$.

\begin{cor}
Let $G$ be a finite group, let $p$ be an odd prime and let $x$ be a $p$-element of $G$. Assume that for every conjugacy class $C$ of $G$, there exists $g\in C$ such that $\la x,g\ra$ is nilpotent. Then $x\in O_p(G)$.
\end{cor}

\begin{proof}
Let $A=\la x\ra$. Then $A$ is a $p$-group and  for every conjugacy class $C$ of $G$, $\la A,g\ra$ is nilpotent for some $g\in C$ by the hypothesis, hence  $A$ is subnormal in $\la A,g\ra$ and so by \cite[Theorem A]{FM}, $A\leq O_p(G)$.
\end{proof}

Example 3.2 in \cite{FM} also shows that the corollary does not hold for $2$-elements.  It is possible that any $2$-elements that satisfy the assumption of the corollary will be in the solvable radical.



\end{document}